\documentclass[preprint,12pt]{elsarticle}

\usepackage{amssymb}

\usepackage{amsthm}

\usepackage{amsmath}
\usepackage{amsfonts}
\usepackage{verbatim}
\usepackage{booktabs} 
\usepackage{graphicx}
\usepackage{float}
\usepackage{hyperref, cases, color, subfigure, times, latexsym, bbm, bm,  pict2e, multirow, ragged2e, epsfig}

\usepackage{enumitem}

\newlist{Henum}{enumerate}{1}
\setlist[Henum,1]{label=\textbf{H\arabic*}, ref=H\arabic*}

\usepackage[colorinlistoftodos]{todonotes}

\newtheorem{thm}{Theorem}
\newtheorem{lem}[thm]{Lemma}
\newtheorem{remark}{Remark}

\begin{document}

\begin{frontmatter}

\title{Traceability of Water Pollution: An Inversion Scheme Via Dynamic Complex Geometrical Optics Solutions}

\author[a,b]{Lingyun Qiu}
\affiliation[a]{organization={Yau Mathematical Sciences Center, Tsinghua University},
            city={Beijing},
            postcode={100084}, 
            country={China}}

\affiliation[b]{organization={Yanqi Lake Beijing Institute of Mathematical Sciences and Applications},
            city={Beijing},
            postcode={101408}, 
            country={China}}

\author[c,d]{Zhongjing Wang}
\affiliation[c]{organization={Department of Hydraulic Engineering, Tsinghua University},
            city={Beijing},
            postcode={100084}, 
            country={China}}

\affiliation[d]{organization={Breeding Base for State Key Laboratory of Land Degradation and Ecological Restoration in Northwest China, Ningxia University},
            city={Yinchuan},
            postcode={750021}, 
            country={China}}

\author[a,b]{Hui Yu}

\author[a]{Shenwen Yu}

\begin{abstract}
We investigate the identification of the time-dependent source term in the diffusion equation using boundary measurements. This facilitates tracing back the origins of environmental pollutants.
Employing the concept of dynamic complex geometrical optics (CGO) solutions, a variational formulation of the inverse source problem is analyzed, leading to a proof of uniqueness result. 
Our proposed two-step reconstruction algorithm first determines the point source locations and subsequently reconstructs the Fourier components of the emission concentration functions.
Numerical experiments on simulated data are conducted. The results demonstrate that the proposed two-step reconstruction algorithm can reliably reconstruct multiple point sources and accurately reconstruct the emission concentration functions.
Additionally, by partitioning the algorithm into online and offline computations, and concentrating computational demand offline, real-time pollutant traceability becomes feasible.
This method, applicable in various fields - especially those related to water pollution, can identify the source of a contaminant in the environment, thus serving as a valuable tool in environmental protection. 
\end{abstract}


\end{frontmatter}

\section{Introduction}\label{section:introduction}

In the field of environmental science, managing water pollution is a significant concern. Water contamination disrupts ecosystems, compromises biodiversity, and poses health risks through fatal waterborne diseases. A pivotal step towards addressing this issue lies in accurately identifying pollution sources, which in turn aids in implementing timely preventative or mitigative actions. 

In this paper, we contribute to this critical endeavor by focusing on an inverse source problem associated with water pollution dispersion. We apply the advection-diffusion-reaction equation to model the pollution dispersion as follows. \begin{equation}\label{full}
    \partial_t u-\kappa \Delta u+v \cdot \nabla u+\gamma u=F(x, t) ,\quad (x,t)\in \Omega \times [0,T],
\end{equation}
where $\Omega$ is the spatial domain, $[0,T]$ is the observation time interval, $\kappa$ is the isotropic diffusion coefficient, $v$ is the convective velocity vector and $\gamma$ is the reaction coefficient. Here, $F(x,t)$ is the pollution source term that varies with space and time, and $u(x,t)$ represents the contaminant's density. 

Our primary objective, which also represents the main challenge of our research, is to determine the source term $F$ using measured boundary data $u|_{\partial \Omega\times[0, T]}$. These measurements span across the full boundary over the observation time interval. Successfully identifying the source term $F$ not only advances the mathematical handling of inverse source problems for parabolic equations but also offers a practical tool for regulators in their fight against water pollution.

Recently, many progresses have been made in the research field on inverse source problems for parabolic equations. The problem involving Dirichlet boundary data is explored in \cite{thanhSpacetimeFiniteElement2021}. The source term is assumed to take the form 
$$
F(x, t)=f(x, t) \phi(x, t)+g(x, t),
$$ where $\phi$ and $g$ are two known functions and $f$, belonging to $L^2(\Omega\times[0,T])$, is the term to be determined. They subsequently derive an optimal control problem with a Tikhonov regularization term to recover $f$. Where $g$ and $\phi$ are provided and $f(x,t)=f(x)$ is a spatial function, the recovery of $F$ is discussed in several works \cite{huangStabilityInverseSource2020, prilepkoInverseSourceProblem2018,husseinDirectInverseSource2020,chengInverseSourceProblem2020, chengContinuationSolutionsElliptic2022}, each based on varying spatial assumptions and measured data. In \cite{huangStabilityInverseSource2020}, the stability of $f$ is constructed by using Carleman estimate. Andrey B. Kostin and his cooperators consider the existence and the uniqueness of the problem in \cite{prilepkoInverseSourceProblem2018,husseinDirectInverseSource2020}, where the measured data is the integral observation $\int_0^T u(t, x) h(t) d t=\varphi(x)$. When $\phi(x, t)=e^{-\gamma t}$ with $\gamma>0$ and $\gamma \neq n^2, n\in \mathbb{N}$, Cheng and liu prove the Hölder stability of $f$ with measurement at two instants on a subdomain of $\Omega$ in \cite{chengInverseSourceProblem2020}. When $\Omega$ is on a hyperplane, the uniqueness of $f$, where the measured data is $u|_{\{\Gamma \subset \partial \Omega\}\times [0,T]}$, is proved in \cite{chengContinuationSolutionsElliptic2022}.  In \cite{linTheoreticalNumericalStudies2022}, given the Neumann boundary data, the authors study the recovery of the semi-discrete source term  
$$F(x, t):=\sum_{k=1}^K p_k(x) \chi_{\left[t_{k-1}, t_k\right)}(t),$$
where $\chi$ is the characteristic function.  Additionally, an inverse random source problem is explored in \cite{wuCarlemanEstimatesStochastic2020}. In general, the aforementioned papers primarily focus on scenarios demonstrating spatial continuity in the source term. Other novel contributions to parabolic inverse source problems also exist. The recovery of reaction coefficient function $\gamma$ is studied in \cite{klibanovConvexificationInverseParabolic2020,bellassouedStablyDeterminingTimedependent2021,chenConvergenceRatesTikhonov2020,fanDeterminingPotentialParabolic2021}. Besides, the initial value problem, that is to recover $u(x,0)$, is discussed in \cite{yangOptimizationMethodMultiparameters2023,linSimultaneousRecoveriesSemilinear2022,casasUsingSparseControl2019a}.  

In the context of addressing water pollution, it is more reasonable to assume that the source term consists of several isolated point sources, whose spatial coordinates are denoted by $s_j$, and the intensity functions depend on time only, characterized by $g_j(t)$. More precisely, we assume that 
\begin{equation}\label{sourceF}
    F(x,t) =  \sum_{j=1}^m g_j(t)\delta (x - s_j).
\end{equation}
Our objective is to determine the number of sources $m$, the intensity functions $g_j(t)$ and the source locations $s_j$. The well-posedness of the direct problem and uniqueness of this inverse problem has been investigated in \cite{AndrleElBadia2012,ElBadiaHa-Duong2002,ElBadiaFarah2010}. To recover the source term, an algebraic identification technique is proposed in \cite{AndrleElBadia2015}. 

Inspired by the widely-utilized adjoint state method \cite{plessixReviewAdjointstateMethod2006}, we consider the weak form of the inverse problem, using some carefully chosen test functions to satisfy the adjoint state equation. These test functions eliminate some interior terms, and the boundary term is obtained from our measurement. The selection of the test functions is crucial to the reconstructive numerical method. In  \cite{AndrleElBadia2015}, the authors use the harmonic polynomials  $v(x,t)= (x_1+{\rm{i}}x_2)^n, n = 0, 1, 2, ...$ in $\mathbb{R}^2$, which are static and independent of time. This will lead to some imperfections in the recovery of intensity functions.
Instead, we adopt the dynamic Complex Geometrical Optics (CGO) solutions.
CGO solutions have been widely applied in inverse problems, as seen in  \cite{uhlmannInverseProblemsSeeing2014, kocyigitApplicationsCGOSolutions2016}. 
To the best of our knowledge, the introduction of dynamic CGO solutions appeared in \cite{choulliLogarithmicStabilityDetermining2018}. These were used to extend the elliptic Carleman estimates in \cite{bukhgeimRECOVERINGPOTENTIALPARTIAL2002} to a hyperbolic Carleman inequality.
To recover the intensity functions, it is advantageous to use test functions in a time-variant form, thus making the dynamic CGO solutions our preferred choice. This method shows significant advantages over existing approaches.

After using the test functions, the weak form of the equation still remains an interior integration term. Boundary control theory is applied for its computation. Theoretically, the existence of boundary control is studied maturely, as seen in \cite{CarthelEtAl1994}. However, many computational difficulties are to be resolved, which are discussed in \cite{munchNumericalApproximationNull2010a}. In this paper, we derive a new approximation way to compute a large number of boundary controls, which both saves computational time and guarantees accuracy.

The paper is organized in the following way. In Section \ref{section:formulation}, the variational form of the inverse source problem using dynamic CGO solutions is derived. The uniqueness and stability results are presented in Section \ref{section:uniqueness}. A numerical method, consisting of a novel approximation method of boundary control and the recovery of the source term, is proposed in Section \ref{section:algorithm}. Section \ref{section:results} presents the numerical results of the algorithm. The conclusion and the future work are discussed in Section \ref{section:conclusion}.

\section{The inverse source problem and its variational formulation}\label{section:formulation}
\subsection{The direct problem}
In this paper, we consider the basic situation where $\kappa$, $\gamma$ and $v$ in \eqref{full} are constants. 
In fact, by introducing the transformation
\[
u(x,t) = \exp\left(\frac{v}{2\kappa}\cdot x-\left(\gamma+\frac{v\cdot v}{4\kappa}\right)t\right) u_1(x,t),
\]
the function $u_1(x,t)$ satisfies the heat equation
\begin{equation*}
    \partial_t u_1 - \kappa\Delta u_1= \exp\left(-\frac{v}{2\kappa}\cdot x+\left(\gamma+\frac{v\cdot v}{4\kappa}\right)t\right)F ,\quad (x,t)\in \Omega \times [0,T],
\end{equation*}

For such a setting, the original problem boils down to the diffusion process of a pollutant.
More precisely, the dynamics of the water pollution is governed by the heat equation in $\mathbb{R}^d$, where $d\in \mathbb{N}$, subject to certain initial and boundary conditions:
\begin{align}\label{directmodel}
\left\{
\begin{array}{ccll}
\partial_t u - \kappa\Delta u&=& F &\text{ in } \,Q, \\
u(x,0) &=& 0 &\text{ in } \,\Omega, \\
\kappa\frac{\partial u}{\partial n} &=& 0 &\text{ on } \Sigma, \\
u(x,t) &=& \varphi(x,t) &\text{ on } \Sigma,
\end{array}
\right.
\end{align}
where $\kappa$ is the chemical diffusivity, $u(x,t)$ is the contaminant concentration, and $F(x,t)$ is the contaminant source.
Here $\Omega \subset \mathbb{R}^d$ stands for a bounded $C^2$-domain, with $n$ representing the outward normal vector to its boundary, $\partial\Omega$. The observation spans the time interval $(0, T)$. The combined space-time domains are represented by $Q = \Omega \times (0,T)$ and $\Sigma = \partial\Omega \times (0,T)$. Specifically, the contaminant source is characterized by:
\begin{align}\label{source}
F(x,t) = F(\theta) = \sum_{j=1}^m g_j(t)\delta (x - s_j),
\end{align}
Here, $m$ denotes the number of sources. The time-dependent intensity and position of the $j$-th point source are represented by $g_j(t)$ and $s_j \in \mathbb{R}^d$, respectively. The comprehensive information about the contaminant source is encapsulated in the set $\theta=\{m,(s_j,g_j)_{1\leq j\leq m}\}$.
Without loss of generality, we assume that

\begin{Henum}
    \item\label{H1}
    The point source positions \( s_j \)'s are distinct.
    \item\label{H2}
    The source intensities \( g_j \in L^2(0,T) \) are nonzero and real functions. Also, they are inactive after a given time \( T^* \in (0,T) \).
    In other words,
    \[
    \exists\, T^* \in (0,T) \text{ such that } g_j(t) = 0 \quad\forall t \in (T^*,T) \text{ and } \forall j \in \{1,2,\ldots, m\}. 
    \]
\end{Henum}
In this paper, we neglect the water depth and consider the inverse source problem in $\mathbb{R}^2$.

As shown in \cite{AndrleElBadia2012}, if $u$ is the solution of \eqref{directmodel} subject to the assumptions \ref{H1} and \ref{H2}, then $\varphi=u|_{\Sigma} \in L^2(\Sigma)$.

\subsection{The inverse problem}

Due to the special form of the source term \eqref{source}, we propose to employ the variational formulation of the problem. The following notations involving $T^*$ will be used:
\begin{align*}
&Q^{-} = \Omega \times (0,T^*), \quad \,\Sigma^{-} = \partial\Omega\times (0,T^*),\\
&Q^{+} = \Omega \times( T^*,T), \quad \Sigma^{+} = \partial\Omega\times (T^*,T).
\end{align*}
The integration with the test function $v(x,t)$ over $Q^{-}$ leads to
\begin{equation}\label{Integration}
    \int_{Q^{-}}\left(\partial_t u-\kappa \Delta u\right) v d x d t=\sum_{j=1}^m \int_0^{T^*} g_j(t) v\left(s_j, t\right) d t .
\end{equation}

Define the test function space:
\begin{align*}
        \mathcal{V} = \big\{ v \in L^2((0,T);H^1(\Omega))\big| \partial_t v + \kappa\Delta v = 0 \text{ in } Q\big\}.
\end{align*}
 
Integrating by parts, we have:

\begin{lem} \label{lemma:Rv}
For any $v \in \mathcal{V}$, we have 
\begin{equation*}
\int_\Omega u(x,T^*)v(x,T^*)\,dx + \kappa\int_{\Sigma^-}\varphi\frac{\partial v}{\partial n}\,d\sigma = \sum_{j=1}^m \int_0^{T^*} g_j(t)v(s_j,t)\,dt.
\end{equation*}
\end{lem}

Define a functional $\mathcal{R}(v)$ by
\begin{equation}\label{functionalR}
\mathcal{R}(v) := \int_\Omega u(x,T^*)v(x,T^*)\,dx + \kappa\int_{\Sigma^-}\varphi\frac{\partial v}{\partial n}\,d\sigma.
\end{equation}

The second term comes directly from measured boundary data $\varphi$. 
We will obtain the first term using the boundary controllability theorem and the details are provided in Section \ref{Sec_bdrycontrol}.
Lemma \ref{lemma:Rv} transforms the inverse source problem into the weak form:
\begin{equation}\label{internal}
    \mathcal{R}(v) = \sum_{j=1}^m \int_0^{T^*} g_j(t)v(s_j,t)\,dt \qquad \forall v \in \mathcal{V}.
\end{equation}

We will demonstrate that a special subset of $\mathcal{V}$ is enough to recover the source.
Given certain test function space $\mathcal{H}\subset \mathcal{V}$, we define the forward operator $\mathcal{P}$ as:
\begin{equation}\label{operator}
    \mathcal{P}_{\mathcal{H}}: \theta=\{m,(s_j, g_j)_{1 \leq j \leq m}\} \to \big(\mathcal{R}(v)\big)_{v \in \mathcal{H}}.
\end{equation}
Then the variational inverse problem can be formulated as follows.

{\it The Variational Inverse Problem: }

Given $\big(\mathcal{R}(v)\big)_{v \in \mathcal{H}}$, find $\theta=
\{m,(s_j, g_j(t))_{1\leq j \leq m}\}$, i.e., 
the number of sources $m$, their fixed positions $s_j$, and their time-dependent intensities $\{g_j\}_{j=1}^m$, such that 
\begin{equation}
\mathcal{P}_{\mathcal{H}}\left(\theta\right)=\big(\mathcal{R}(v)\big)_{v \in \mathcal{H}}.
\end{equation}

\subsection{The computability of \texorpdfstring{$\mathcal{R}(v)$}{Rv}}\label{Sec_bdrycontrol}
Note that $u(x,T^*)$ may not be available in applications. In order to evaluate $\mathcal{R}(v)$ through the measured boundary data $\varphi$, we need to transform the domain integral $\int_\Omega u(x,T^*)v(x,T^*)\,dx$ to a boundary integral. We employ the boundary controllability theory to do this; see more details in \cite{CarthelEtAl1994}.
For the sake of completeness, we state the main results in this section.
In fact, we only need to evaluate 
\begin{equation}\label{uvstar}
\int_\Omega u(x,T^*)v(x,T^*)\,dx \quad \text{ for any given test function } v(x,t) \in \mathcal{H}.
\end{equation} 
We formulate a suitable boundary controllability problem to obtain the above integration. 
It is stated in two lemmas below.
\begin{lem}\label{lemmabdrycontrol}
    For each given $p\in H^1(\Omega)$, there exists a pair of functions $(\psi, \omega)$ with $\psi \in L^2\big((T^*,T);H^{1}(\Omega)\big)$, $\omega \in L^\infty(\Sigma^+)$ which solves the following system
\begin{subequations}\label{sys_psi_ini}
\begin{numcases}{}
\partial_t\psi + \kappa \Delta\psi = 0 & in $\,Q^+$,\\  
\kappa\frac{\partial \psi}{\partial n} = \omega & on  $\Sigma^+$, \\
\psi(x,T) = 0 & in  $\,\Omega$, \\
\psi(x,T^*) = p & in  $\,\Omega$.
\end{numcases}
\end{subequations}
\end{lem}
\begin{proof}
    According to the Theorem 6.1 of Chapter 2 in  \cite{Fursikov2000}, the following equations has a solution $\omega$:
    $$
    \left\{\begin{array}{llll}
    \partial_t\psi_1 + \kappa \Delta\psi_1 = 0 & \text{in} \,Q^+,\\  
    \kappa\frac{\partial \psi_1}{\partial n} = \omega & \text{on}  \,\Sigma^+, \\
    \psi_1(x,T) = -\psi_2(x,T) & \text{in}  \,\Omega, \\
    \psi_1(x,T^*) = 0 & \text{in}  \,\Omega,
    \end{array}\right.
    $$
    where $\psi_2$ is the solution to
    $$
    \left\{\begin{array}{llll}
    \partial_t\psi_2 + \kappa \Delta\psi_2 = 0 & \text{in} \,Q^+,\\
    \kappa\frac{\partial \psi_2}{\partial n} = 0 & \text{on}  \,\Sigma^+, \\
    \psi_2(x,T^*) = p & \text{in}  \,\Omega.
    \end{array}\right.
    $$
    Then $\psi=\psi_1+\psi_2$ and $\omega$ are the solution to \eqref{sys_psi_ini}.
\end{proof}

Using the boundary control theory, the computability of $\mathcal{R}(v)$ is given in \cite{AndrleElBadia2015}. We restate the lemma in our setting and provide a brief proof.

\begin{lem}\label{lemma:control}
    For a test function $v\in\mathcal{V}$, take $p=v(x,T^*)$ in \eqref{sys_psi_ini} and $\omega$ is the corresponding control, and we have:
    \begin{equation}\label{RvCompt}
    \mathcal{R}(v) = 
    \kappa\int_{\Sigma^-}\varphi\frac{\partial v}{\partial n}\,d\sigma
    + 
    \int_{\Sigma^+}\varphi\omega\,d\sigma. 
    \end{equation}
\end{lem}
\begin{proof}
     Multiplying the equation in the system \eqref{sys_psi_ini} with the solution $u$, and integrating by parts in $Q^+$ leads to
    \begin{align*}
        0 =& \int_{Q^+}\big(\partial_t\psi + \kappa \Delta\psi\big)u\,dxdt\\
        =&\int_{Q^+}\Big[\partial_t(\psi u) - \psi\partial u\Big]\,dxdt
        +\int_{\Sigma^+}\kappa\left(\frac{\partial \psi}{\partial n} u - \psi\frac{\partial u}{\partial n}\right)\,d\sigma
        +\int_{Q^+} \kappa\psi\Delta u \,dxdt \\
        =&\left.\int_{\Omega}u\psi\,dx\right|_{T^*}^T + \int_{\Sigma^+}\kappa\frac{\partial \psi}{\partial n} u\,d\sigma \\ 
        =& -\int_{\Omega} u(x,T^*) v(x,T^*)\,dx + \int_{\Sigma^+}\omega \varphi\,d\sigma.
    \end{align*}
    Recalling the definition of $\mathcal{R}(v)$ from \eqref{functionalR}, the proof is completed.
\end{proof}
\begin{remark}
    The boundary control $\omega$ that satisfies \eqref{sys_psi_ini} is not necessarily unique. However, the value of $\mathcal{R}(v)$ remains invariant to the specific choice of $\omega$. This invariance can be attributed to the fact that, given two boundary controls $\omega_1$ and $\omega_2$ associated with the test function $v$, the relation 
    $$
    \int_{\Sigma^+}\varphi(\omega_1-\omega_2)\,d\sigma=0
    $$ 
    holds true.
\end{remark}

\subsection{Choice of test functions}\label{subsection:fourier}
There exist various types of test functions that satisfy the adjoint problem, including harmonic polynomial functions of the form: 
$$
v(x)=(x\cdot a)^n \qquad \text{ for } x=(x_1,\cdots , x_d)\in\mathbb{R}^d, a=(a_1,\cdots, a_d) \in \mathbb{C}^d, a\cdot a=0.
$$ 
In \cite{AndrleElBadia2015}, it was investigated that harmonic polynomial functions are utilized to identify the information of sources. However, their algebraic method is prone to instability due to the need for calculating matrix eigenvalues. Another drawback is that as the power $n$ increases, the function becomes numerically unstable. This instability directly affects the accuracy of the numerical computation of $R(v)$, thereby potentially disrupting the entire algorithm. Additionally, harmonic polynomials are independent of time. Consequently, from $R(v)$, we can only gather information regarding the total intensity $\int_{0}^{T*} g_j(t) dt$. To recover the intensity functions, an alternative method must be employed, such as solving the minimization problem described in \cite{AndrleElBadia2015}. This method will involve solving heat equations for each iterative step, thus leads to large computational time. 

To enhance the numerical method, we incorporate the dynamic complex geometrical solutions (CGO) solutions \( v(x,t) = e^{\alpha t + \rho \cdot x} \), where \( \alpha \in \mathbb{C} \) and \( \rho \in \mathbb{C}^d \). First, choosing different dynamic CGO solutions lead to the stability of the numerical method as affirmed by Theorem~\ref{thm: stab}. Moreover, the time component of dynamic CGO solutions aids in recovering the Fourier coefficients of intensity functions, as indicated by the subsequent Theorem~\ref{thm:fourier}.

Define the function space \( \mathcal{H} \subset \mathcal{V} \) and the subspaces \( \mathcal{H}_k \) of \( \mathcal{H} \) by
\begin{align*}
        \mathcal{H} = \big\{ v \in L^2((0,T);H^1(\Omega))\big|& v(x,t) = e^{\alpha t+\rho \cdot x} 
        \text{ with } \alpha \in \mathbb{C}, \rho \in \mathbb{C}^d\\
        &\text{ satisfies } \alpha + \kappa\rho\cdot\rho = 0.\big\},
\end{align*}
\begin{align*}
        \mathcal{H}_k = \big\{ v \in L^2((0,T);H^1(\Omega))\big|& v(x,t) = e^{-\frac{2k\pi {\rm{i}}}{T^*} t+\rho \cdot x} 
        \text{ with } \rho \in \mathbb{C}^d\\
        &\text{ satisfies } -\frac{2k\pi {\rm{i}}}{T^*} + \kappa\rho\cdot\rho = 0.\big\}.
\end{align*}
In the context of our problem, it's crucial to consider the frequency-dependent behavior of the system. By introducing subspaces \( \mathcal{H}_k \), we aim to capture the essence of this frequency-dependent behavior. Each subspace \( \mathcal{H}_k \) focuses on a specific frequency \( k \), allowing us to isolate and analyze the system's response at that particular frequency. This decomposition into subspaces is not only mathematically elegant but also computationally efficient. The forthcoming algorithm, which solves the subproblem for each frequency \( k \), can then leverage the information contained within the corresponding subspace \( \mathcal{H}_k \), ensuring a more targeted and precise numerical solution.

For the integer $k$, let $\alpha_k = -\frac{2k\pi {\rm{i}}}{T^*}$, and define the corresponding Fourier coefficients of $g_j$ by 
\begin{equation}\label{def_lambda}
\displaystyle \lambda_j(\alpha_k) = \int_0^{T^*} g_j(t)e^{\alpha_k t}\,dt.
\end{equation}

\begin{thm}\label{thm:fourier}
Given $\big(m,(s_j)_{1\leq j \leq m}\big)$, the $k$-th Fourier coefficients $\lambda_j(\alpha_k)$ can be calculated from $\{\mathcal{R}(v)|v\in \mathcal{H}_k\}$.
\end{thm}
Since the proof involves Theorem \ref{thm: eigenvalue}, we postpone it in Section \ref{subsection:matrix}.

Noticing that the intensity functions $g_j(t) \in L^2(0,T^*)$, the Fourier transform tells us:
    \begin{equation}\label{fouriers}
        g_j(t) \approx \frac{1}{T^*}\sum_{k}\lambda_j(\alpha_k)e^{-\alpha_kt} \qquad \text{ for } j = 1, \ldots, m. 
    \end{equation}
\section{The well-posedness and stability of the inverse problem}\label{section:uniqueness}
In this section, we address the uniqueness of the solution to the inverse problem based on the new data type $\{\mathcal{R}(v)|v\in \mathcal{H}\}$. Additionally, we discuss the feasibility of our proposed numerical method. The analyses of these two main results rely on some certain constructive dynamic CGO solutions. Let $\rho =a+b\mathrm{i}$ with $a, b \in \mathbb{R}^d$. 
We will construct several dynamic CGO solutions, $v(x,t)= e^{\alpha t+(a+b{\rm{i}})\cdot x}$, based on the real vector $a$. This construction is detailed in the following lemma:
\begin{lem}\label{lemma:vectorb}
    Fix $k\in \mathbb{R}$ and for real vector $a\in \mathbb{R}^d$ with $|a| \geq \sqrt{|k\pi|}$, there exists a real vector $b\in \mathbb{R}^d$ such that the complex vector $\rho =a+b\mathrm{i} \in \mathbb{C}^d$ satisfies $\rho \cdot \rho =2k\pi \mathrm{i}$.
\end{lem}

\begin{proof}
    Suppose $a=|a|(\cos \beta a_1,\sin \beta)$, where $a_1\in \mathbb{R}^{d-1}$ with $|a_1|=1$. If $|a| \neq 0$, we simply take $b=|a|(\cos \gamma a_1,\sin \gamma)$, where $\gamma=\beta+\arccos \left(\frac{k\pi}{|a|^2}\right)$.
    If $|a|=0$, i.e., $k=0$, then we take $b=\boldsymbol{0}$.
    
\end{proof}
\subsection{Uniqueness theorem}
The uniqueness theorem, which maps the measured boundary data $\varphi$ to the source set  $\theta$ under assumptions \ref{H1} and \ref{H2}, is established in \cite[Theorem 1]{AndrleElBadia2015}.
However, our numerical approach transforms the boundary data $\varphi$ into a different data type, $\{\mathcal{R}(v)|v\in \mathcal{H}\}$. Given that we employ a smaller test function space $\mathcal{H}$, it becomes essential to establish the uniqueness of the mapping from $\{\mathcal{R}(v)|v\in \mathcal{H}\}$ to the source term $\theta$. The following theorem addresses this:

\begin{thm}
   The operator $\mathcal{P}_{\mathcal{H}}$ defined by \eqref{operator} is an injection.
\end{thm}

\begin{proof}\
We will show that there exists only a trivial solution $\theta$ such that
    \begin{equation}\label{eqn_zero}
        \mathcal{R}(v)=0 \qquad \forall v \in \mathcal{H}
    \end{equation}
by contradiction. 
For $m=1$, it is clear that $\theta$ must be a trivial solution and hence $\mathcal{P}_{\mathcal{H}}$ is an injection.
Assume that we have more than one nonzero sources.
    Without loss of generality, order the sources in terms of locations such that 
    $$
    |s_1|\geq |s_j|\qquad \forall 2\leq j \leq m.
    $$
Since all the locations are distinct, we have:
    \begin{equation}\label{src}
    s_1\cdot s_1 > s_{1}\cdot s_j  \qquad \forall 2\leq j \leq m. 
    \end{equation}
For each $k$, we employ Lemma \ref{lemma:vectorb} to construct the test functions
\[
v_{\xi}(x,t) = e^{-\alpha_k t}w_{\xi}(x) = e^{-\alpha_k t}e^{\xi(s_1 + b_{\xi}\mathrm{i})\cdot x} \qquad \forall \xi > \xi_0 := \frac{1}{|s_1|}\sqrt{\frac{|k|\pi}{T^*}}.
\]
Note that $|w_{\xi}(x)| = |w_{\xi_0}(x)|^{\frac{\xi}{\xi_0}}$, and the inequality \eqref{src} implies that 
\begin{equation}\label{leqeq}
|w_{\xi_0}(s_1)| > |w_{\xi_0}(s_j)| \qquad \forall 2 \leq j \leq m.
\end{equation}
Obviously $v_{\xi} \in \mathcal{H}$. 
Plugging the test function $v_{\xi}$ into \eqref{eqn_zero} and dividing both sides by $w_{\xi}(s_1)$ lead to 
    \begin{align*}
        0=&\sum_{j=1}^{m} \int_0^{T^*} g_j(t)e^{-\alpha_kt}\,dt \frac{w_{\xi}(s_j)}{w_{\xi}(s_1)}\\
        =& \int_0^{T^*} g_1(t)e^{-\alpha_kt}\,dt + \sum_{j=2}^{m}\int_0^{T^*} g_j(t)e^{-\alpha_kt}\,dt \frac{w_{\xi}(s_j)}{w_{\xi}(s_1)}.
    \end{align*}
It follows that
    \begin{align*}
        \left|\int_0^{T^*} g_1(t)e^{-\alpha_kt}\,dt\right| &= \left|-\sum_{j=2}^{m}\int_0^{T^*} g_j(t)e^{-\alpha_kt}\,dt\frac{w_{\xi}(s_j)}{w_{\xi}(s_1)}\right|\\
         &\leq \sum_{j=2}^{m}\left|\int_0^{T^*} g_j(t)e^{-\alpha_k t}\,dt\right| \left|\frac{w_{\xi_0}(s_j)}{w_{\xi_0}(s_1)}\right|^{\frac{\xi}{\xi_0}}.
    \end{align*}
Taking the limit as $\lambda \to +\infty$, together with \eqref{leqeq}, we conclude that
\[
\int_0^{T^*} g_1(t)e^{-\alpha_kt}\,dt = 0 \qquad \forall k \in \mathbb{R}. \quad
\Longrightarrow \quad g_1(t) \equiv 0 \qquad \forall t \in [0,T^*].
\]
It contradicts the assumption \ref{H2}, and the proof is complete.
\end{proof}

\begin{remark}
    Lemma \ref{lemma:control} shows that the measured data $\varphi$ uniquely determines the values of functional $\mathcal{R}(v)$. Consequently, the injectivity of $\mathcal{P}_{\mathcal{H}}$ supports \cite[Theorem 1]{AndrleElBadia2015}, which states that the source set $\theta$ is uniquely determined by the boundary measurement $\varphi_{\theta}$. 
\end{remark}

\subsection{The feasibility of the numerical method and stability results}\label{subsection:matrix}

 In the following discussion, we explore the reconstruction of the Fourier coefficients associated with the intensity functions. Building upon the forthcoming Section~\ref{subsection:minimization}, which introduces an efficient numerical method to ascertain both the number of sources and their respective locations, we assume the availability of the values $\{m, (s_j)_{j=1}^m\}$ within the set $\theta$. Our analysis will particularly emphasize the solvability of the terms $\lambda_j(\alpha_k)$.

Fix $k$, and suppose there exist $m$ distinct complex vectors $\{\rho_l^{(k)}\}_{1\leq l \leq m}$ associated to $\alpha_k$ according to Lemma \ref{lemma:vectorb} such that
\[
v_{k,l}(x,t) = e^{\alpha_k t + \rho_l^{(k)}\cdot x} \in \mathcal{H}.
\]
Then equations \eqref{internal} and \eqref{def_lambda} lead to 
\begin{equation}\label{weakformavg}
    \mathcal{R}(v_{k,l}) = \sum_{j=1}^m \lambda_j(\alpha_k) e^{\rho_l^{(k)}\cdot s_j}.
\end{equation}
We end up with a linear system of the Fourier coefficients for the $k$-th mode: 
\begin{equation}\label{linear_lambda}
A_k \Lambda_k = \mathcal{R}_k,
\end{equation}
where $A_k = \left[a^{(k)}_{l,j}\right]_{1\leq l, j \leq m}$ with $a^{(k)}_{l,j}=e^{\rho_l^{(k)}\cdot s_j}$, and the two vectors are
\[
\Lambda_{k}=\bigg(\lambda_{1}(\alpha_{k}), \cdots, \lambda_{m}(\alpha_{k})\bigg)^{T}, \quad \mathcal{R}_{k}=\bigg(\mathcal{R}(v_{k, 1}), \cdots, \mathcal{R}(v_{k, m})\bigg)^{T}.
\]
The solvability of $\Lambda_k$ is equivalent to the invertibility of the matrix $A_k$.
We will show that for properly chosen test functions, i.e., complex vectors $\{\rho_l^{(k)}\}_{1\leq l \leq m}$, all the eigenvalues of $A_k$ are nonzero, and hence $A_k$ is invertible.
We start with a lemma to estimate the magnitude of the eigenvalues of a matrix.
 
\begin{lem}\label{lemma:circle}
   Consider a general matrix $A=(a_{l,j})_{1\leq l, j \leq m}$ and its eigenvalue $\mu$. For any given positive real numbers  $p_1,\ldots, p_m$, we have 
    \begin{equation*}
    	\mu \in \bigcup_{j=1}^{m}\mathcal{B}\left(a_{j,j},p_j\sum_{l\neq j}\frac{|a_{l,j}|}{p_l}\right)
    \end{equation*}
	where the open ball $\mathcal{B}(z_0,r):=\{z\in \mathbb{C}:|z-z_0|<r\}$.
\end{lem}
\begin{proof}
    Introduce a matrix
    \begin{equation}\label{matrixs}
        B=D^{-1}AD = \left[b_{l,j}\right]_{1\leq l, j \leq m} = \left[\frac{p_j}{p_l}a_{l,j}\right]_{1\leq l, j \leq m} ,
    \end{equation}
    where $D={\rm{diag}}(p_1,\cdots,p_m)$.
    Since $A$ and $B$ are similar, they have the same eigenvalues.
     The result is obtained by applying the Gershgorin circle theorem in \cite{brualdiRegionsComplexPlane1994} to $B^{T}$.
\end{proof}

We now show that the open balls of $A_k$ exclude the origin, implying that all eigenvalues are distinct from zero.
\begin{thm}\label{thm: eigenvalue}
    Given the set $(m, (s_j)_{j=1}^m)$ and an integer $k$,  there exist test functions $v_{k,1},\cdots,v_{k,m}\in \mathcal{H}_k$ , such that the matrix $A_k$ is invertible.
\end{thm}
\begin{proof}
By shifting the domain, we can guarantee that the source locations $s_j$ are distinct from the origin. We then arrange them in the order:
\[
|s_1| \geq |s_2| \geq \cdots \geq |s_m| > 0.
\]
Given their distinctness, we can identify two positive constants, \(\delta\) and \(r\), satisfying:
\begin{equation}\label{distinction}
|s_j-s_l| \geq \delta > \sqrt{\frac{2\ln(m-1)}{r}} \qquad \forall j \neq l.
\end{equation}
Define $r_k$ as 
$$
r_k:= \max \left\{1, r,  \frac{1}{|s_m|}\sqrt{\frac{|k|\pi}{T^*}}\right\}.$$ 
Referring to Lemma \ref{lemma:vectorb}, a complex vector \(\rho_l\) exists such that: 
    \[
        \rho_l = r_k s_l+\mathrm{i}b_l,\quad\text{ and }\quad \rho_l\cdot \rho_l = \frac{2k\pi \mathrm{i}}{T^*}. 
    \]
    The associated test function is \(v_{k,l}(x,t) = e^{\alpha_kt + \rho_l \cdot x}\), and the matrix \(A_k = [a_{l,j}]_{1\leq l, j \leq m}\) with \(a_{l,j} = e^{\rho_l \cdot s_j}\).
    Here we omit the superindex $k$ in $\rho_l$ and $a_{l,j}$ for concise presentation.

    Selecting \(p_j = e^{\frac{r_k}{2}|s_j|^2}\) and for any \(z\) in the open ball \(\mathcal{B}(a_{j,j}, p_j\sum_{l\neq j}\frac{|a_{l,j}|}{p_l})\), we deduce:
    \begin{align*}
        |z| &\geq|a_{j,j}|- p_j\sum_{l\neq j}\frac{|a_{l,j}|}{p_l} \\
            & = \exp(r_k |s_j|^2) - \exp\left(\frac{r_k}{2}|s_j|^2\right)\sum_{l\neq j}\exp\left(\frac{r_k}{2}(2s_l\cdot s_j -|s_l|^2)\right) \\
            & = \exp\left(r_k |s_j|^2 \right)\left(1 - \exp\left(-\frac{r_k}{2}|s_j|^2\right)\sum_{l\neq j}\exp\left(\frac{r_k}{2}(2s_l\cdot s_j -|s_l|^2)\right)\right) \\
            & = \exp\left(r_k |s_j|^2\right)\left(1 - \sum_{l\neq j} \exp\left(-\frac{r_k|s_j-s_l|^2}{2}\right)\right) \\
            & \geq \exp\left(r_k |s_j|^2\right)\left(1-(m-1)\exp\left(-\frac{r_k\delta^2}{2}\right)\right) \\
            & \geq 1-(m-1)e^{-\frac{r\delta^2}{2}}.
    \end{align*}
    To summarize, we have
    \begin{equation}\label{rad}
        |z| \geq C(m;\delta,r):= 1-(m-1)e^{-\frac{r\delta^2}{2}}>0 \qquad \forall \gamma\in \mathcal{B}_j \text{ for } 1\leq j \leq m.
    \end{equation}
    Given that any eigenvalue $\mu$ of $A_k$ resides in one of the aforementioned open balls (as per Lemma \ref{lemma:circle}), $\mu$ is distinct from zero. Consequently, $A_k$ is invertible.
\end{proof}

Now we are ready to provide a proof of Theorem \ref{thm:fourier} as follows.

\begin{proof}[Proof of Theorem \ref{thm:fourier}]
    For each integer $k$, choose the test functions $v_{k,l},$ for $l = 1, ..., m$ as described in Theorem \ref{thm: eigenvalue}. Then the resulting linear system $A_k \Lambda_k = \mathcal{R}_k$ possesses a unique solution $\Lambda_k$ since $A_k$ is invertible guaranteed by Theorem \ref{thm: eigenvalue}.
\end{proof} 

For $k=0$, the estimate on the eigenvalues in Theorem \ref{thm: eigenvalue} leads to a stability result concerning the total intensities of the sources. 
\begin{thm}\label{thm: stab}
    Given the set $(m, (s_j)_{j=1}^m)$, there exist test functions $v_{0,l}(x)$ for $1\leq l \leq m$ such that the resulting linear system $A_0\Lambda_0 = \mathcal{R}_0$ is stable in the sense that
    \begin{equation}
        ||\Lambda_0-\widehat{\Lambda_0}||_2\leq \frac{1}{C(m;\delta,r)}||\mathcal{R}_0-\widehat{\mathcal{R}_0}||_2,
    \end{equation}
where the pairs $(\Lambda_0, \mathcal{R}_0)$ and $(\widehat{\Lambda_0}, \widehat{\mathcal{R}_0})$ are associated to two sets of source intensities $(g_j)_{j=1}^m$ and $(\widehat{g_j})_{j=1}^m$, respectively.
\end{thm}
\begin{proof}
Construct the test functions $v_{0,l}$ analogously to Theorem \ref{thm: eigenvalue}, i.e., $\rho_l = r_0s_l + \mathrm{i}b_l$.
We aim to show that $A_0$ is a Hermitian matrix, i.e. $$\rho_l\cdot s_j=\overline{\rho_j}\cdot s_l\quad \text{ for } 1\leq l, j \leq m.$$
Thanks to Lemma \ref{lemma:vectorb}, we have for $1\leq l \leq m$,
\begin{equation*}
\rho_l \cdot \rho_l=0.
\quad \Longleftrightarrow \quad
|b_l|=r_0|s_l|, b_l \cdot s_l=0.
\end{equation*}
In fact, $b_l$ is a (clockwise) rotation of $r_0s_l$ with an angle of $\pi/2$.
It follows that
    \begin{equation*}
        (b_l+b_j)\cdot(s_l+s_j)=0 \qquad \text{ for } 1\leq j, l \leq m.
    \end{equation*}
Thus, \(b_l\cdot s_j + b_j \cdot s_l = 0\) for \(1 \leq j, l \leq m\). Further, we can express 
\begin{align*}
\rho_l\cdot s_j - \overline{\rho_j}\cdot s_l
=& (r_0s_l + \mathrm{i}b_l)\cdot s_j - (r_0s_j - \mathrm{i}b_j)\cdot s_l \\
=& r_0s_l\cdot s_j - r_0 s_j\cdot s_l + \mathrm{i}(b_l\cdot s_j + b_j\cdot s_l) = 0. 
\end{align*}
Consequently, $A_0$ is Hermitian, and so is its inverse $A_0^{-1}$. Letting $\mu$ represent the eigenvalue of $A_0$, by \eqref{rad}, we have
\[
\|A_0^{-1}\|_2 = \max_{\mu} \left|\frac{1}{\mu}\right| \leq \frac{1}{C(m;\delta, r)}.
\]
This directly yields
\begin{equation*}
\|\Lambda_0-\widehat{\Lambda_0}\|_2 = \|A_0^{-1}(\mathcal{R}_0-\widehat{\mathcal{R}_0})\|_2 \leq \frac{1}{C(m;\delta, r)}\|(\mathcal{R}_0-\widehat{\mathcal{R}_0})\|_2.
\end{equation*}
\end{proof}

\begin{remark}
    The assertions of Theorem \ref{thm: stab} still hold if test functions are chosen such that the associated matrix $A_0$ is normal.
\end{remark}

\section{Numerical algorithm}\label{section:algorithm}
This section introduces the proposed reconstruction algorithm. Assumption~\ref{H2} confirms that all intensity functions $g_j(t)$ are real, leading to the relation
$$\Lambda_k=\bar{\Lambda}_{-k}.$$
Consequently, it suffices to recover $\Lambda_k$ for $k \geq 0$. 
Using the linear system for the Fourier coefficients of the $k$-th mode as described in \eqref{linear_lambda}, we formulate the inversion of the number of sources $m$, the Fourier coefficients of intensity functions $\Lambda_{k}$, and the source locations $S = (s_1, \ldots, s_m)^T$ as the following optimization problem:
\begin{equation}\label{minmization}
    \min _{m, \, S, \, \Lambda_{k}} \left(\sum_{0\leq k \leq K}\left\|A_k( m,S)\Lambda_{k}-\mathcal{R}_{k}\right\|^2_2 \right).  
\end{equation}

To enhance efficiency and accuracy, we split this problem into two sub-problems. Drawing inspiration from the invertibility established in Theorem~\ref{thm: eigenvalue} for $k=0$ and the stability result in Theorem~\ref{thm: stab}, we first determine the number of sources, their locations, and total intensities by addressing:
\begin{equation}\label{lambda_zero}
    (m, S, \Lambda_0) = \operatorname*{\arg\min}_{ m\leq M, S, \Lambda_0} \left\|A_0( m, S )\Lambda_{0}-\mathcal{R}_{0}\right\|^2_2. 
\end{equation}

Assuming the number of sources does not exceed $M$, we employ $3M$ test functions $v_{0,l}, 1\leq l \leq 3M$, which are defined as: 
\begin{equation*}
v_{0,l}(x,t)=\exp(\rho_{0,l} \cdot x).    
\end{equation*}
\begin{equation*}
\rho_{0,l}=\frac{2\sqrt{2}}{\textrm{diag}(\Omega)}\bigg((\cos{\beta_l},\sin{\beta_l})+\mathrm{i}(\cos{\gamma_l},\sin{\gamma_l})\bigg), \beta_l=\frac{2j\pi}{M}, \gamma_l=\beta_l-\frac{\pi}{2},
\end{equation*}

Subsequently, we isolate the reconstruction process for the Fourier coefficients based on the value of $k$, reconstructing the coefficients for each frequency by resolving:
\begin{equation}\label{lambda_nonzero}
A_k(m,S)\Lambda_{k}=\mathcal{R}_{k}, \quad k = 1,2,\dots , K.
\end{equation}

For $k>0$, we only need no more than $M$ test functions and we will give the definition of them in Section \eqref{subsection:fourier inverse}.

We can now outline the algorithm as follows:
\begin{enumerate}
\item[(i)] Given two integers $K$ and $M$, construct the CGO-type test functions $v_{0,l}, 1\leq l \leq 3M$ and  $v_{k,l}, 1\leq k \leq K, 1\leq l \leq M$. Compute the approximation to the boundary control $\omega$ using the bases $\{\omega_\eta\}$ given in Sec. \ref{subsection: boundary control}. 
\item[(ii)] Compute the functional $\mathcal{R}_0 = \big(\mathcal{R}(v_{0,l})\big)_{1\leq l \leq 3M}$ and $\mathcal{R}_k = \big(\mathcal{R}(v_{k,l})\big)_{1\leq l \leq M}$ in Equation \eqref{RvCompt}. Solve the minimization problem \eqref{lambda_zero} to obtain the set $(m, S, \Lambda_0)$.
\item[(iii)] With known $(m, S)$, solve the linear systems \eqref{lambda_nonzero} to obtain $\Lambda_k, 1\leq k\leq K$, and then perform the inverse Fourier transform to retrieve $g_j, 1\leq j \leq m$.
\end{enumerate}

Next we will address the implementation of each step.

\subsection{Efficient computation of boundary control $\omega$ and $\mathcal{R}(v)$} \label{subsection: boundary control}

For our numerical method, a substantial number of calculations for $\mathcal{R}(v)$ are required. Solving the minimization problem of boundary control \eqref{sys_psi_ini}, as proposed in \cite{CarthelEtAl1994}, can be computationally intensive as it involves solving the initial-boundary value problem of the heat equations repeatedly.
To mitigate this computational demand, especially when computing a large number of boundary controls, we introduce a novel method that enhances efficiency in these scenarios.
This method centers around the construction of a set of basis functions to approximate the space $L^2(\Sigma^+)$.
Each basis function within this set is characterized by a triplet of indices, $\eta = (\eta_1, \eta_2, \eta_3)$ with $\eta_j = 1, 2, ..., W_j, j=1,2,3$, where $W_j$'s are given positive integers.

Specifically, for a spatial domain $\Omega=[0,L_x]\times[0,L_y]$, we define $\omega_\eta \in L^2(\Sigma^+)$ as
\begin{equation*}
\omega_\eta(x,y,t)=\cos\left(\frac{2\eta_3-1}{2}\pi\frac{t-T^*}{T-T^*}\right)
\left\{\begin{array}{ll}
-\kappa_2\sin\left(\frac{\eta_1\pi}{L_x}x\right) \frac{\eta_2\pi}{L_y}& \text{ for } y = 0, \\
\kappa_2\sin\left(\frac{\eta_1\pi}{L_x}x\right) \frac{\eta_2\pi}{L_y}(-1)^{\eta_2}& \text{ for } y = L_y, \\
-\kappa_1\sin\left(\frac{\eta_2\pi}{L_y}y\right) \frac{\eta_1\pi}{L_x}& \text{ for } x = 0, \\
\kappa_1\sin\left(\frac{\eta_2\pi}{L_y}y\right) \frac{\eta_1\pi}{L_x}(-1)^{\eta_1}& \text{ for } x = L_x. 
\end{array}\right.
\end{equation*}

For each basis function $\omega_{\eta}$, there is an associated function $\psi$ leading to $\psi_{\eta}(T^*)$, which can be determined using a standard PDE (Partial Differential Equation) solver for (\ref{sys_psi_ini}a-c). 
With this set of functions $\{\psi_{\eta}(T^*)\}$, we can solve the minimization function with respect to $v(T^{*})$ as
\begin{equation}\label{controlapproximation}
\{a_{\eta}\} = \operatorname*{\arg\min}_{\{\hat{a}_{\eta}\}} \left\|\sum_{\eta}\hat{a}_{\eta}\psi_{\eta}(T^{*})-v(T^{*})\right\|^2_{L^2(\Omega)} .   
\end{equation}
The approximation of boundary control is then given by $\omega \approx \sum\limits_{\eta}a_{\eta}\omega_{\eta}$.

Our methodology offers a significant reduction in computational time by requiring the PDE solver only a limited number of times to determine $\psi_{\eta}(T^*)$. The function approximation problem \eqref{controlapproximation} is a typical least-squares problem and several methods such as QR decomposition are derived for it. We specifically employ the LSQR method for an efficient solution. 

Crucially, the computation of the bases $\psi_{\eta}(T^*)$ is independent of both the boundary measurement $\varphi$ and the test function $v$. 
This ensures that their precomputation, albeit initially time-intensive, becomes a one-time offline process. As a result, not only do subsequent changes in the test functions lead to minor online computational adjustments, but modifications in the measured data also demand minimal online recalculations. This highlights a dual advantage: for a given spatial domain, after performing the offline computation once, we can conduct multiple measurements and run the inversion—which requires the boundary controls—rapidly. 
This level of efficiency is a marked improvement over the traditional methods described in \cite{CarthelEtAl1994}, highlighting the advantages and versatility of our proposed approach.

After introducing our new methodology for efficient boundary control computation, we can outline the specific steps as follows:
\begin{itemize}
    \item[(i-1)] Given each boundary control basis $\omega_{\eta}$, obtain $\psi_{\eta}(T^*)$ from the system (\ref{sys_psi_ini}a-c).
    \item[(i-2)] Given the test functions $v_{0,l}, 1\leq l \leq 3M$, solve \eqref{controlapproximation} by the LSQR method.
    \item[(i-3)] Obtain $\sum _{\eta}a_{\eta}\omega_{\eta}$ as the approximation to the boundary control $\omega$.
\end{itemize}
For $1\leq k \leq K$, the computation of the boundary control of $v_{k,l}$ can be performed analogously. Also, it is worth noting that step (i) is performed offline.

\subsection{Identification of source number $m$ and source location $S$}\label{subsection:minimization}

In this section, we discuss the process of recovering source number $m$, source location $S$ and the zeroth mode Fourier coefficients $\Lambda_0$. We begin with substituting the boundary control obtained in step (i) into \eqref{RvCompt} to obtain $\mathcal{R}_0$. Our next objective is to solve \eqref{lambda_zero}. This task presents two major challenges. Firstly, the source number $m$ remains unknown, demanding a strategy for its identification. Secondly, the gradient-based algorithm to solve \eqref{lambda_zero} is prone to convergence towards undesirable local minima.

To mitigate these difficulties, we construct the following numerical algorithm.
\begin{itemize}
 
     \item[(ii-1)] Use the boundary control from step (i) to compute $\mathcal{R}_0$. 
     \item[(ii-2)] Initialize the source number as $M$ in \eqref{lambda_zero}. Subsequently use the Sequential Least SQuares Programming(SLSQP) method to solve \eqref{lambda_zero}.
    \item[(ii-3)] Use the random sampling method to bypass the local minimum. In practice, we examine several random initial guesses for the source location and the zeroth intensity. The solution with the minimum relative error from these trials is accepted as the genuine solution.
    \item[(ii-4)] Assess the results. If two points are near each other or certain points exhibit significantly low intensities, some redundancy or non-existence of the points is implied. If $N$ denotes the number of these extraneous points, then $M$ is adjusted as $M=M-N$. We then revert to the previous step (ii-2). Otherwise, the algorithm finishes in the next step (ii-5).
    \item[(ii-5)] Then the source number is defined as $m=M$, and the source position is determined as $S$.
\end{itemize}

In step (ii-2), fixing the source count aids in translating the minimization problem to one involving continuous variables, paving the way for gradient-based optimization techniques like SLSQP. The relatively modest complexity of equation \eqref{lambda_zero} ensures a fast minimization, with the repeated minimizations in step (ii-3) causing only a negligible increase in computational time. Step (ii-4) is important, offering a systematic strategy for identifying the source number, seamlessly aligning with practical scenarios.  

\subsection{Computation of Fourier coefficients for $k>0$}\label{subseciont: recover intensity}
In Section~\ref{subsection:minimization}, we obtain the source number $m$, the source location $S$ and the zeroth mode coefficients $\Lambda_0$. Based on the inverse Fourier transform, we are ready to state the first intensity recovery method, referring to it by Inverse Fourier Method. 

From a theoretical standpoint, for any $k>0$, there exist $m$ test functions $v_{l}(x,t) \in \mathcal{H}_k$ such that the matrix $A_k=(v_{l}(s_j,0))_{1\leq l , j \leq m}$ is invertible. This is proved in Theorem \ref{thm: eigenvalue}. With \eqref{linear_lambda}, we can obtain the $k$-th mode Fourier coefficients. 

In our numerical method, the intensity function's recovery hinges on the inverse Fourier transform, which might induce the Gibbs phenomenon near a jump discontinuity at $T^*$. To mitigate this, we extend $g_j$ to be periodic over $[-T^*,T^*]$ and employ the test functions:
\begin{equation}\label{extend}
v_{k}(x,t)=\exp(\rho_{k} \cdot x)\exp\Big(-\frac{k\pi \rm{i}}{T^{*}}t\Big).    
\end{equation}

For $k\geq 1$, at most $M$ test functions $v_{k,l}, 1\leq l \leq M$ are required. The spatial vectors $\rho_{k,l}$ are defined as:
\begin{equation*}
\rho_{k,l}=r_k\bigg((\cos{\beta_l},\sin{\beta_l})+\mathrm{i}(\cos{\gamma_l},\sin{\gamma_l})\bigg), \beta_l=\frac{2l\pi}{M}, \gamma_l=\beta_l-\arccos{\frac{k\pi}{2r_k^2 T^{*}}},
\end{equation*}
where $r_k$ is a parameter depends on the size of space domain $\Omega$ and the index of the Fourier mode $k$. 
If $r_k$ is excessively large, the minimization problem and the computation of $\mathcal{R}(v)$ become unstable. Conversely, if $r_k$ is too small, the definition of $\rho_{k,l}$ is not valid. Typically, we select:
\begin{equation*}
r_k=\max\left\{\frac{2\sqrt{2}}{\textrm{diag}(\Omega)}, \sqrt{\frac{k\pi}{2T^*}}\right\}.
\end{equation*}
Analogous to \eqref{fouriers}, the inversion of intensity functions is based on:
\begin{equation*}
    g_j(t) \approx \frac{1}{2T^*}\sum_{k}\lambda_j\left(\frac{k\pi}{T^*}\right)e^{-\frac{k\pi{\rm{i}}}{T^*}t} \qquad \text{ for } j = 1, \ldots, m. 
\end{equation*}
\begin{remark}
    The test functions used in Section \ref{subsection:minimization} and \ref{subseciont: recover intensity} are independent of the exact source locations $S$ and the source number $m$. This independence allows for the advance preparation of boundary controls for these test functions.
\end{remark}
\subsection{Approximation of intensity functions}
To improve the accuracy of the recovery results, we introduce an alternative method for the reconstruction of the intensity functions, referring to it by Approximation Method. Similar to Section \ref{subseciont: recover intensity}, the source number $m$ and source location $S$ are obtained in Step (ii). With this information, we aim to choose several bases of intensity functions, denoted as $h_l$ for $0 \leq l \leq 2L$ with $L$ being a positive integer. Suppose $f_{l,j}(x,t)=\delta(x-s_j)h_l(t)$ and we can define the function $\varphi_{l,j}$ that satisfies the direct problem
\begin{align}\label{basis of intensity}
\left\{
\begin{array}{ccll}
\partial_t u - \kappa\Delta u&=& f_{l,j} &\text{ in } \,Q, \\
u(x,0) &=& 0 &\text{ in } \,\Omega, \\
\kappa\frac{\partial u}{\partial n} &=& 0 &\text{ on } \Sigma, \\
u(x,t) &=& \varphi_{l,j} &\text{ on } \Sigma.
\end{array}
\right.
\end{align}
From the above system we obtain the bases $\varphi_{l,j}$ for the boundary measurement $\varphi$. Consequently, the following minimization problem arises: 
\begin{equation}\label{gtapproximation}
\{b_{l,j}\}=\operatorname*{\arg\min}_{\{\hat{b}_{l,j}\}} \left\|\sum_{0\leq l \leq 2L, 1\leq j \leq m}\hat{b}_{l,j}\varphi_{l,j}-\varphi\right\|^2_{L^2(\Sigma)}, 
\end{equation}
The recovery of intensity functions is then expressed as
\begin{equation}\label{gt_approximation}
    g_j(t)=\sum_{l=0}^{2L}b_{l,j}h_l(t).
\end{equation}

In real scenarios, the source intensity functions are non-negative functions. Therefore, post-reconstruction necessitates the truncation of the negative segment.
Given that $\hat{g}(t)$ represents the intensity function procured via the aforementioned recovery technique in \eqref{gt_approximation}, the definitive result $g(t)$ is formulated as
\begin{equation*}
g(t)=
\left\{\begin{array}{ll}
\hat{g}(t) & \text{ if } \hat{g}(t) \geq 0, \\
0& \text{ if } \hat{g}(t) < 0, 
\end{array}\right.
\qquad \text{for} \quad 0 \leq t \leq T^*.
\end{equation*}
In practice, the basis functions in \eqref{gt_approximation} are chosen as
\begin{equation*}
    h_{2l}=\cos\Big(\frac{l\pi }{T^*}t\Big), \quad h_{2l+1}=\sin\Big(\frac{l\pi }{T^*}t\Big) \qquad \text{ for } l = 0, \ldots, L.
\end{equation*}
To summarize, the algorithm presented in this subsection follows these steps:
\begin{itemize}
    \item [(iv-1)] Set $L$ and determine the intensity function basis $h_l$ for $0\leq l\leq 2L$.
    \item [(iv-2)] With $m$ and $S$ obtained in Step (ii), solve the heat equations \eqref{basis of intensity} for $0\leq l \leq 2L$.
    \item [(iv-3)] Use LSQR to solve the problem \eqref{gtapproximation}.
    \item [(iv-4)] Reconstruct the nonnegative intensity functions by \eqref{gt_approximation} and the post-reconstruction. 
\end{itemize}
\begin{remark}
    The Approximation Method needs extra online computational time compared with method in Section \ref{subseciont: recover intensity}. This is due to the necessity of solving the heat equations \eqref{basis of intensity} $(2L+1)m$ times.
\end{remark}
\section{Numerical simulations}\label{section:results}
For all numerical tests, we set $\Omega = [0,1000]\times [0,1000]$ m$^2$ with a constant diffusion coefficient $\kappa = 1$ m$^2$/s. 
The inactive moment $T^*$ is set to 38.4 hours (equivalent to 138240 seconds) while the total observation duration $T$ is 48 hours (equivalent to 172800 seconds). The intensity functions utilized in the numerical tests, when expressed in seconds, are as follows:

\begin{equation*}
\begin{aligned}
& q(t)=\frac{1}{2}\Bigg[1-\tanh \bigg(\frac{t-0.9 T^*}{21600} \bigg)\Bigg];  \\
& g_1(t)=\left[3+\frac{3}{2} \sin \left(\frac{2 \pi t }{T^*} \right)\!+\!3 e^{-\left(\frac{10}{T^*}(t-36000)\right)^2}\!+\!\frac{5}{2} e^{-\left(\frac{15}{T^*}(t-100500)\right)^2}\right] q(t); \\
& g_2(t)=\Bigg[5+\sin \left(\frac{4 \pi t}{T^*}\right)\Bigg] q(t); \qquad g_3(t)=6\left[1-\left(\frac{t}{T^*}\right)^2\right] q(t);\\
&  g_4(t)=5 q(t);\qquad g_5(t)=g_6(t)=4 q(t).
\end{aligned}    
\end{equation*}

\subsection{Boundary control reconstruction and calculation of $\mathcal{R}(v)$}
    The approximation method in Section \ref{subsection: boundary control} is used to calculate boundary control $\omega$. We choose the basis function index $W_1=W_2=W_3=10$. Additionally, the maximal LSQR iteration step in Step (i-2) is defined to be $1000$. The numerical results of the approximation of $v(T^*)$ and corresponding $\mathcal{R}(v)$ are shown in Figure \ref{pic:rv}. The relative error of $v(T^*)$ is given by $\frac{||v(T^*)-\tilde{v}||_{L^2(\Omega)}}{||v(T^*)||_{L^2(\Omega)}}$ and the relative error of $\mathcal{R}(v)$ is given by $\frac{|\mathcal{R}(v)-\tilde{\mathcal{R}}(v)|}{|\mathcal{R}(v)|}$, where $\Tilde{v}$ and $\Tilde{\mathcal{R}}(v)$ are the numerical approximations. The approximations of $v(T^*)$ and $\mathcal{R}(v)$ show a good agreement, which are both $\mathcal{O}(10^{-3})$. This demonstrates the effectiveness of our new method to compute the boundary control.
\begin{figure}[t]
		\centering
		\begin{minipage}{0.48\textwidth}
			\subfigure[Relative error of $v(T^*)$ ]{\includegraphics[width=0.98\columnwidth]{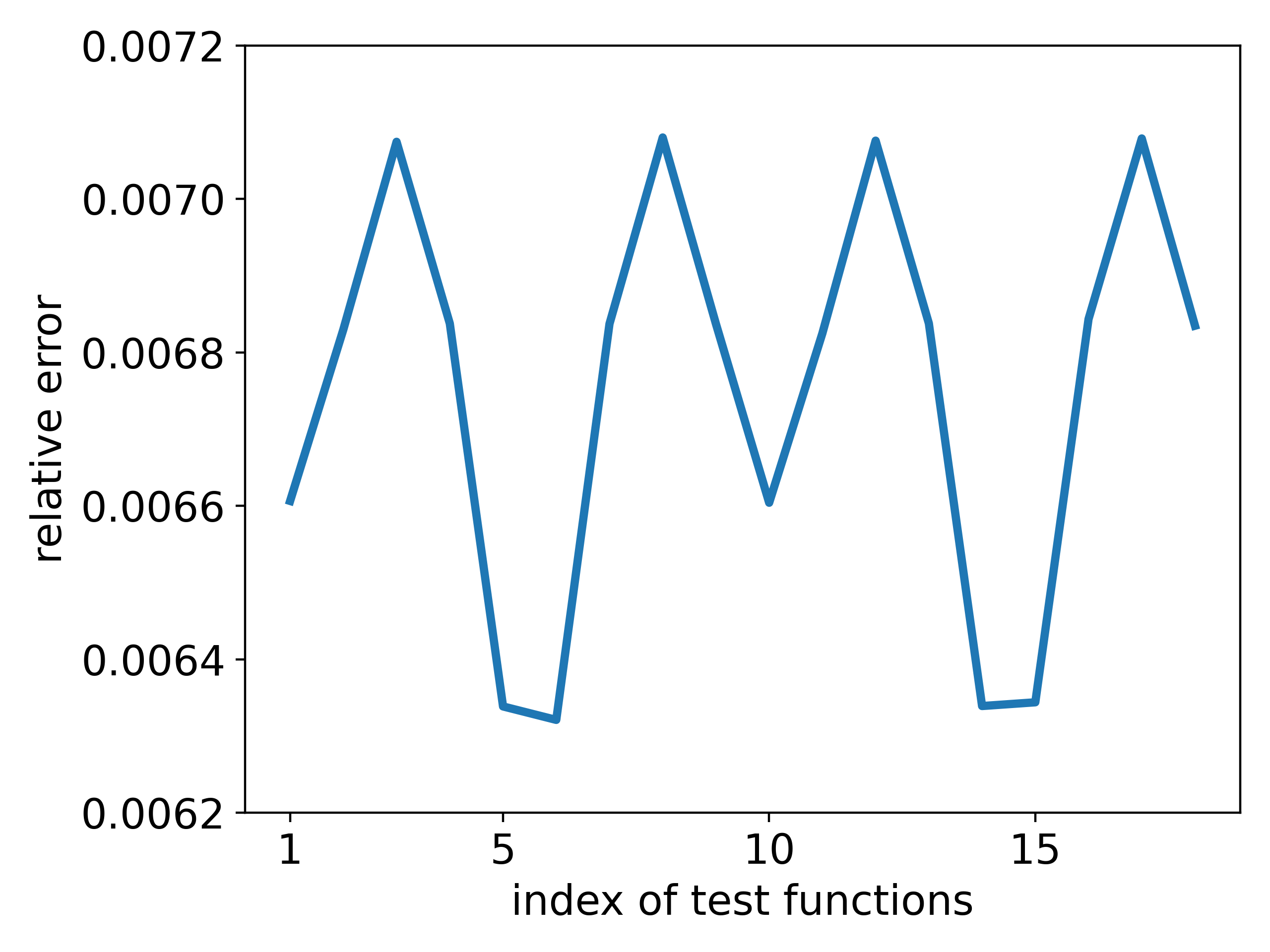}}\label{subfig:rv}
		\end{minipage}
	\begin{minipage}{0.48\textwidth}
			\subfigure[Relative error of $\mathcal{R}(v)$]{\includegraphics[width=0.98\columnwidth]{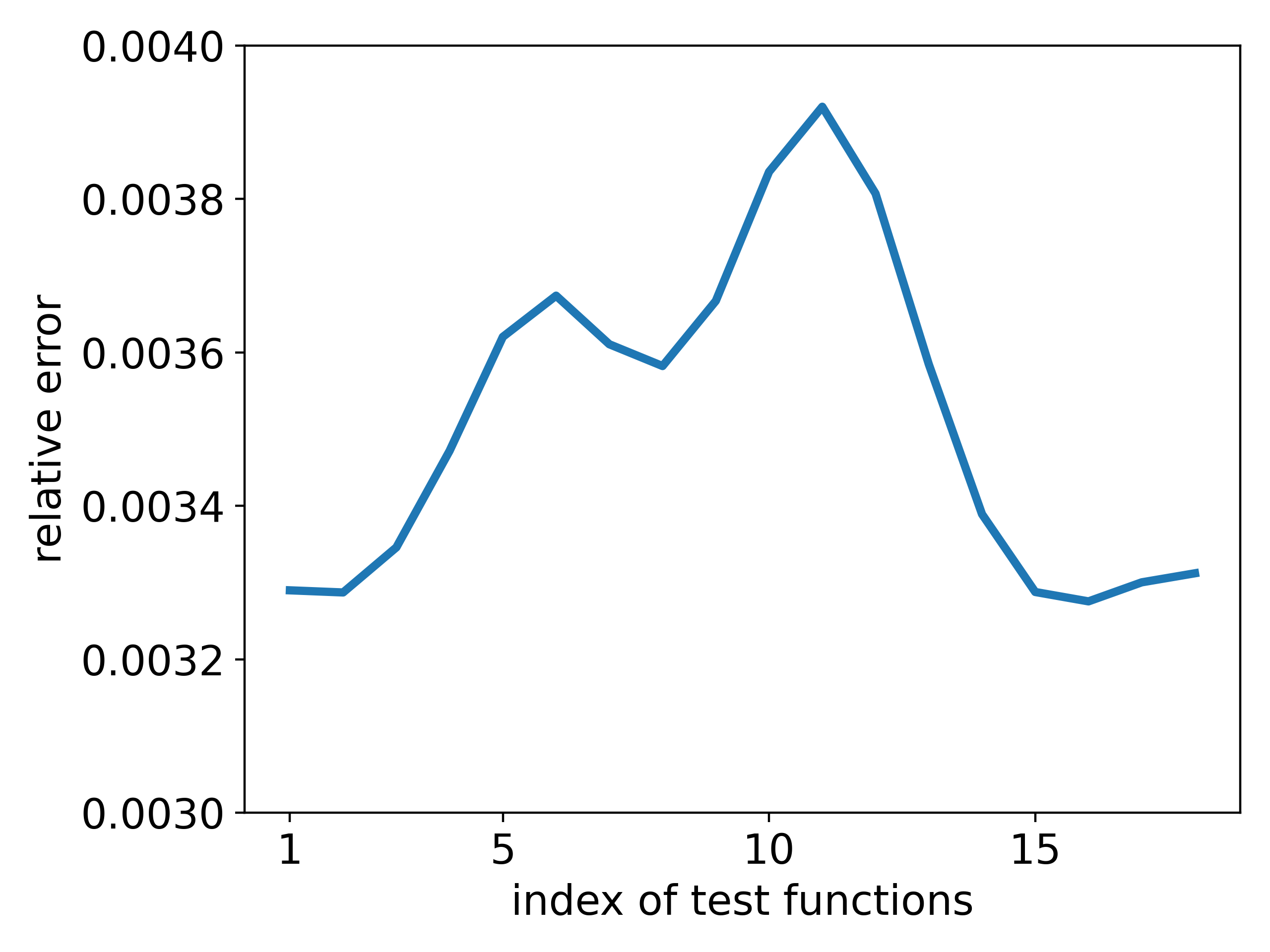}}\label{subfig:rverror}
		\end{minipage}
		\caption{Approximation Results of Boundary Control}
		\label{pic:rv}
\end{figure}
\subsection{The Identification of the Source Locations}\label{subsection: location}
We perform numerical simulations for Step (ii) to demonstrate the accuracy of the approximation on $m$ and $S$. For Step (ii-1), the initial guess of source number $M$ is set to 7 and we choose 50 random initial guesses in Step (ii-3). The recovery results of source locations are presented in Figure \ref{pic: position}. The position error is referred to as $\frac{||S-S_{\text{Iter}}||_2}{||S||_2}$, and the intensity error is referred to as $\frac{||\Lambda-\Lambda_{\text{Iter}}||_2}{||\Lambda||_2}$. The numerical results show that the algorithm correctly identifies the source number $m$. 
The recoveries of the source location $S$ and total intensities are also highly precise. 
For the two-point case in Figure \ref{pic: position}a, the reconstruction method handles the situation where the distance between two sources is 150 m. Generally, as the source number $m$ increases, the source number is accurately determined. But the reconstruction becomes more unstable, particularly for the total intensities of the sources. Additionally, we observe that the sources near the boundary are better recovered than those in the middle of the whole domain.

\begin{figure}[h]
        \centering
        \subfigure[Position error: 0.2\%, intensity error: 1.7\%.]{\includegraphics[width=0.50\hsize, height=0.40\hsize]{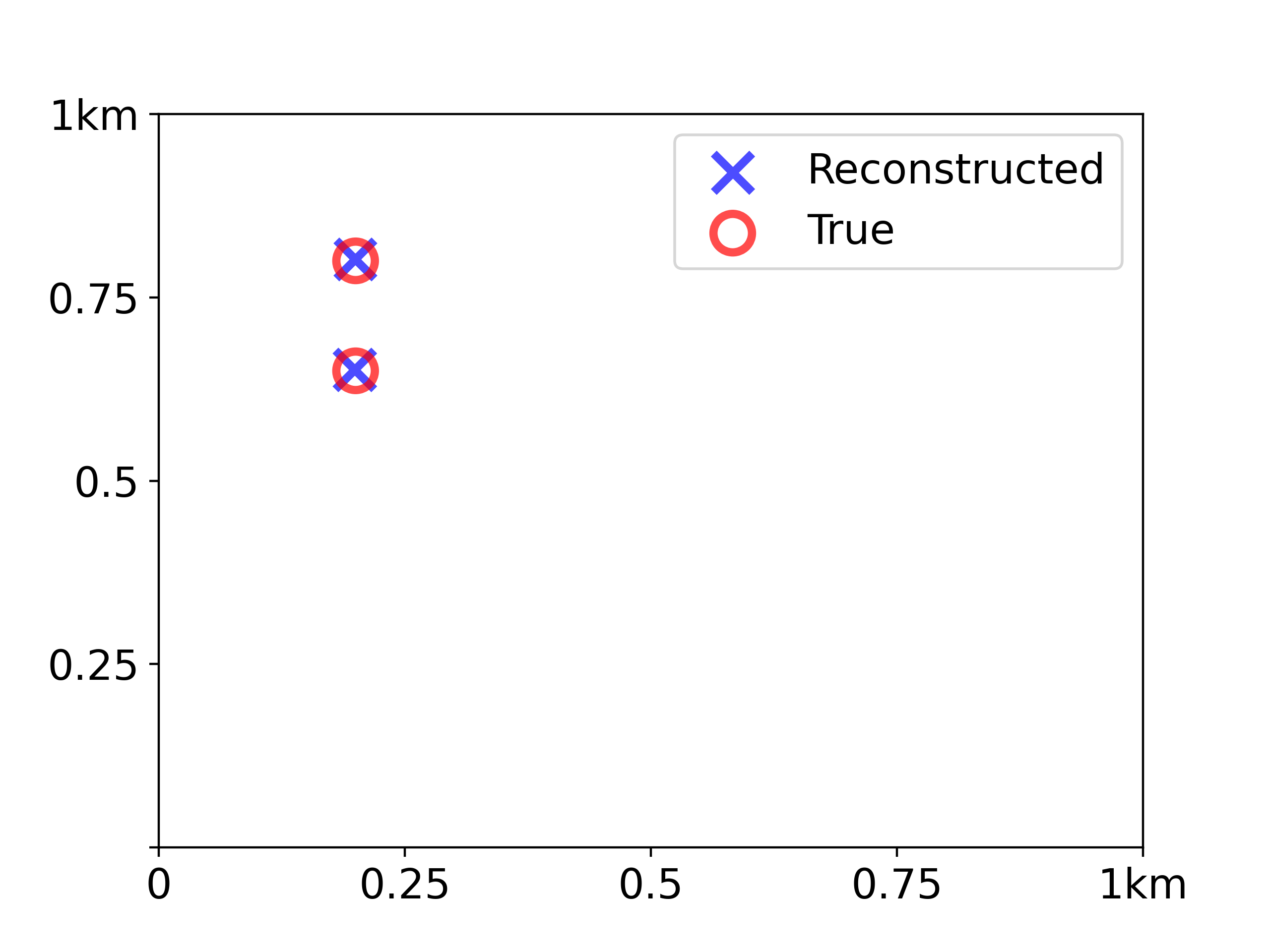}}\label{subfig:2points}
        \hspace{-3ex}
        \vspace{-2ex}
        \subfigure[Position error 0.3\%, intensity error 2.7\%.]{\includegraphics[width=0.50\hsize, height=0.40\hsize]{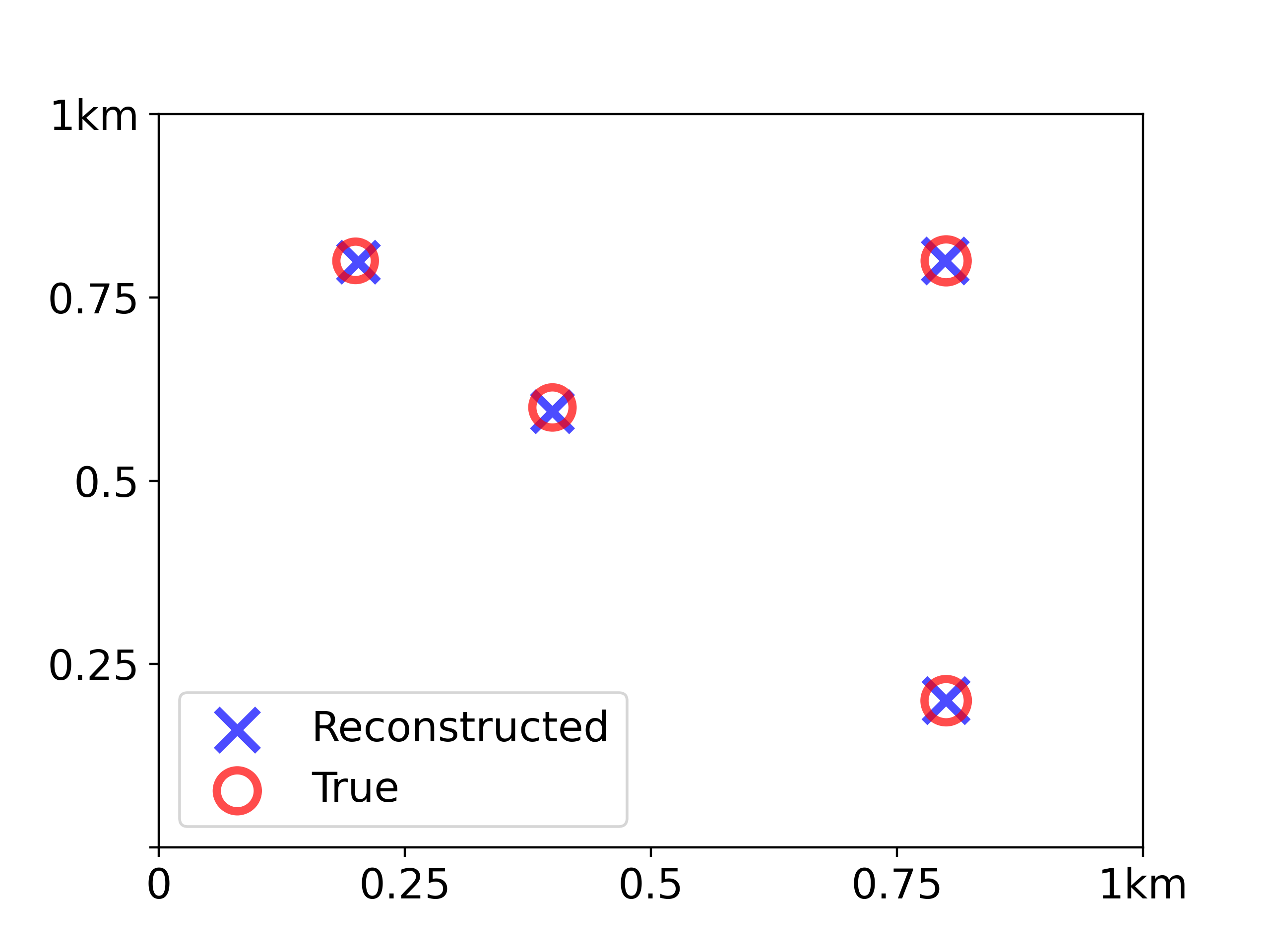}}\label{subfig:4points}
        \subfigure[Position error 0.4\%, intensity error 3.0\%.]{\includegraphics[width=0.50\hsize, height=0.40\hsize]{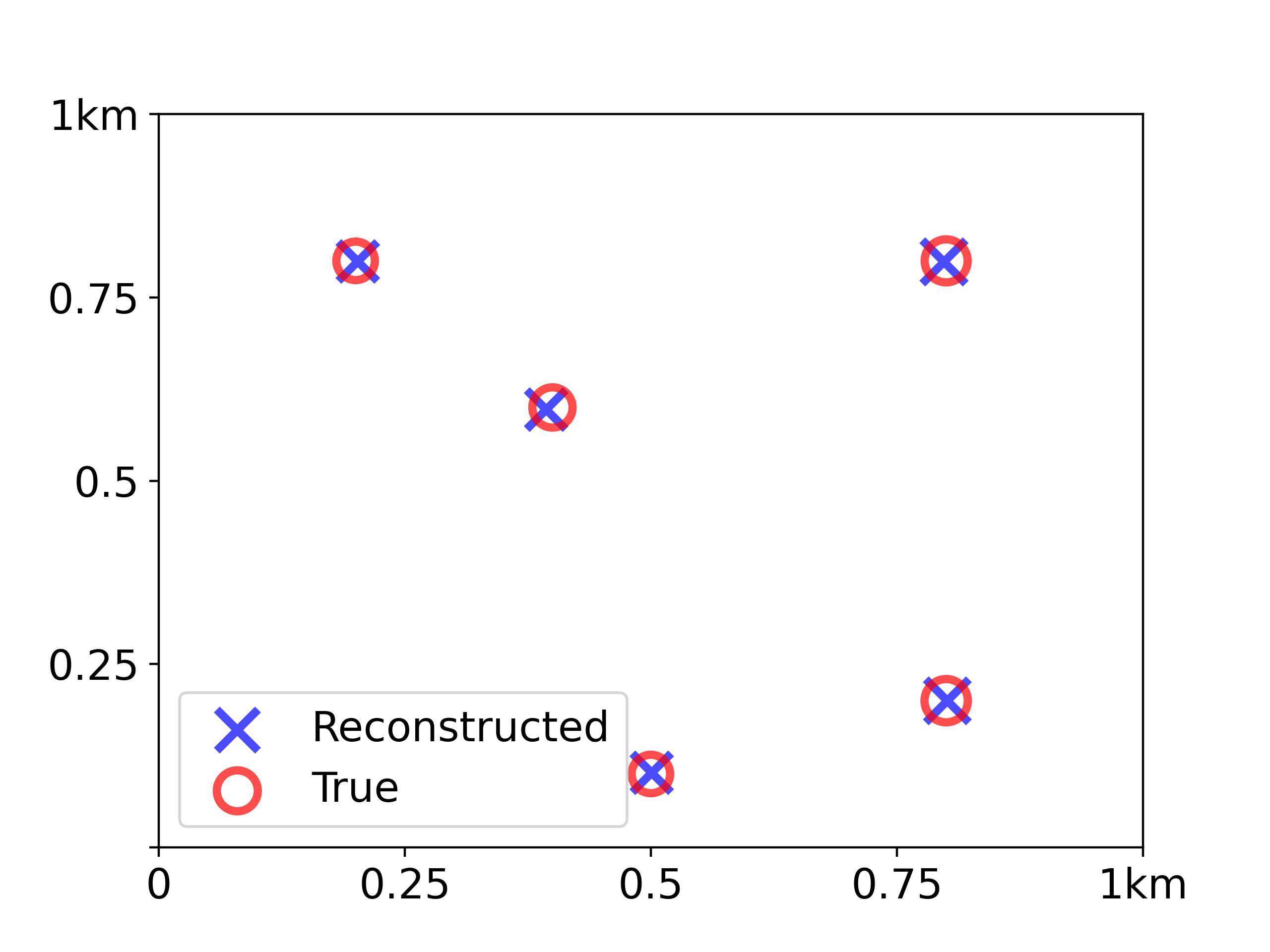}}\label{subfig:5points}
        \hspace{-3ex}
        \vspace{-2ex}
        \subfigure[Position error 6.7\%, intensity error 29.0\%.]{\includegraphics[width=0.50\hsize, height=0.40\hsize]{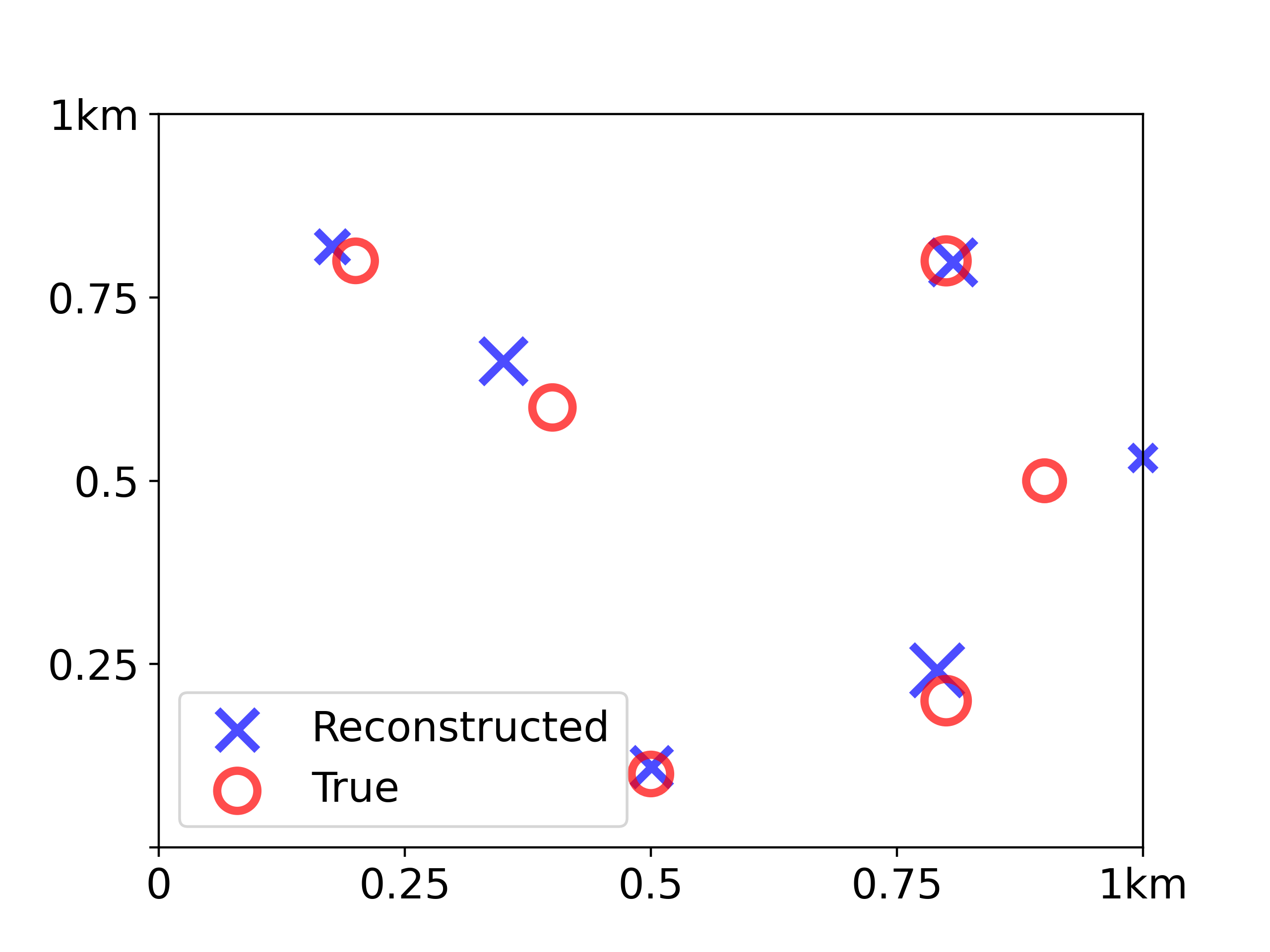}}\label{subfig:6points}
        \caption{Reconstruction of positions}
        \label{pic: position}
\end{figure}

\begin{table}[h]
    \label{table:6points}
    \centering
    \caption{Reconstruction of positions}
    \resizebox{13.5cm}{!}{
    \begin{tabular}{cccccccccccccccc} 
    \toprule 
    &&&& \multicolumn{2}{c}{Source 1}& \multicolumn{2}{c}{Source 2}&\multicolumn{2}{c}{Source 3}&\multicolumn{2}{c}{Source 4}&\multicolumn{2}{c}{Source 5}&\multicolumn{2}{c}{Source 6}\\
    \hline 
    \multirow{2}*{2 points case}&&\multicolumn{2}{c}{True position}&200 &800 &200 &650&&&&&&&&\\
    ~&&\multicolumn{2}{c}{Reconstructed position}&200 &802 &199 &651&&&&&&&&\\
    \hline\hline\noalign{\smallskip}
    \multirow{2}*{4 points case}&&\multicolumn{2}{c}{True position}&200 &800 &800 &200 &400 &600 &800 &800&&&&\\  
    ~&&\multicolumn{2}{c}{Reconstructed position}&203 &798& 800& 200& 400& 594& 799 & 800&&&&\\
    \hline\hline\noalign{\smallskip}
    \multirow{2}*{5 points case}&&\multicolumn{2}{c}{True position}&200 &800 &800 &200 &400 &600 &800 &800&500&100&&\\   
    ~&&\multicolumn{2}{c}{Reconstructed position}&202 &799& 801& 200& 394& 597& 798 & 799&501&102&&\\
    \hline\hline\noalign{\smallskip}
    \multirow{2}*{6 points case}&&\multicolumn{2}{c}{True position}&200 &800 &800 &200 &400 &600 &800 &800&500&100&900&500\\   
    ~&&\multicolumn{2}{c}{Reconstructed position}&176 &819& 790& 241& 350& 663& 807 & 798&501&109&1000&531\\
    \bottomrule
    \end{tabular}}
\end{table}

\subsection{Recovery of intensity functions' Fourier coefficients}\label{subsection:fourier inverse}

Leveraging the source locations of $5$-points case, discussed in Section~\ref{subsection: location}, where the intensity functions are given by $g_1(t), \cdots, g_5(t)$, we continue to present our algorithm. 
The recovery method, grounded in Fourier coefficients, specifies that $|k| \leq K = 8$ during Step (iii), followed by truncating any negative components. 
The results from this inversion process can be seen in Figure \ref{fig: figure5} under the title `I-F Method' and the $L^2$ relative errors of the recovered intensity functions are tabulated in Table \ref{table:intensity functions}.

The outcomes demonstrate a satisfactory recovery of the intensity functions. The two peaks of both $g_1$ and $g_2$ are accurately determined. 
Apart from $g_3(t)$, the relative $L^2$ errors of the other four intensity functions remain less than $10\%$. 
Such observations underscore the notion that the accuracy in intensity function recovery is closely linked to the precision of source location recovery. 
Specifically, given that the accuracy of the location $s_3$ is slightly poorer than the others, the resulting intensity function exhibits a notable disagreement. Additionally, it is noteworthy from Figure \ref{fig: figure5} that the recovery performance at the endpoint surpasses that at the starting point. This discrepancy can be associated with the extension phase of the intensity functions during Step (iii).

\begin{figure}[htp]
    \centering
    \subfigure{\includegraphics[width=.23\hsize, height=0.28\hsize]{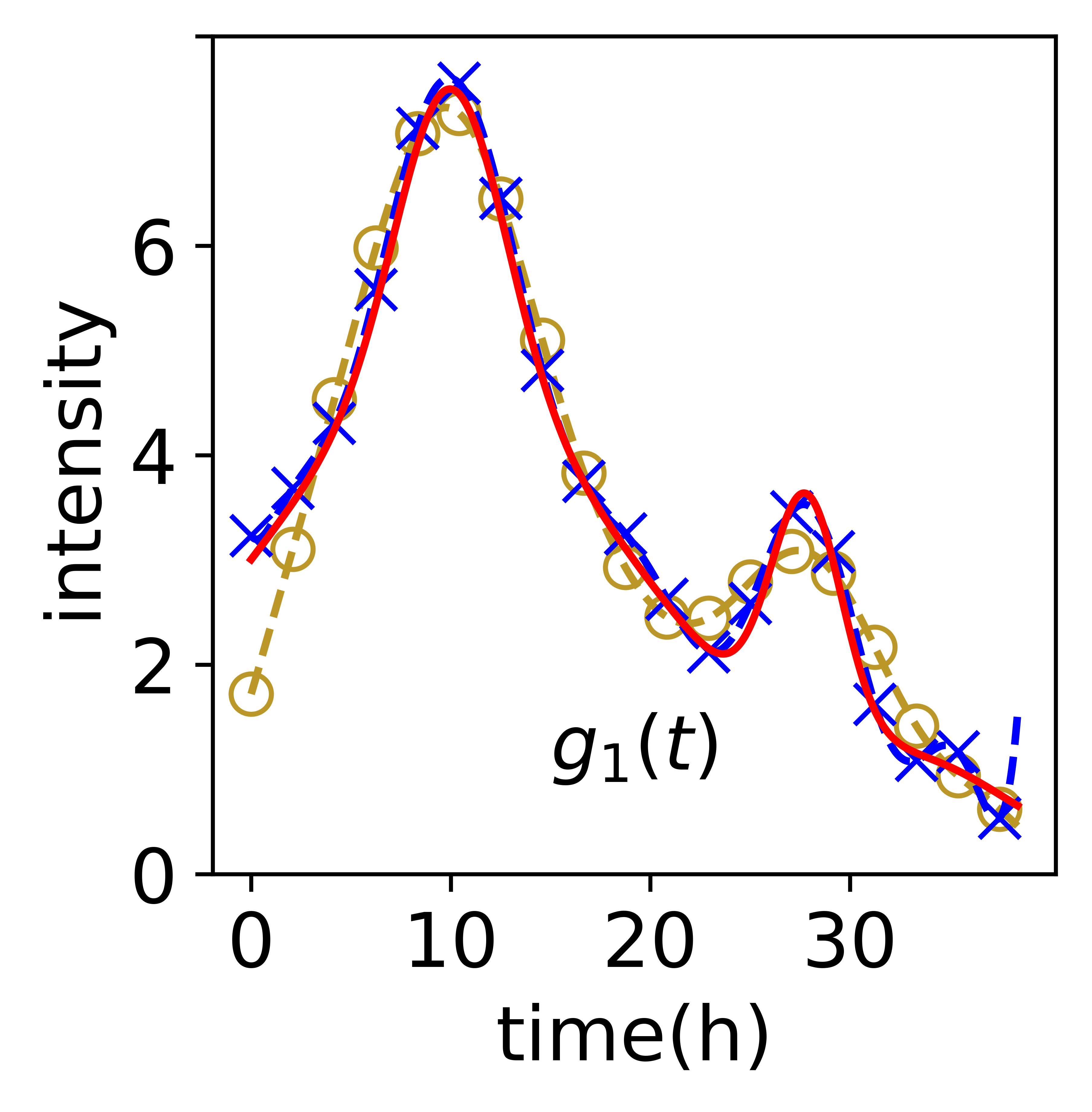}\label{fig: sub_figure1}}%
    \subfigure{\includegraphics[width=0.1925\hsize, height=0.28\hsize]{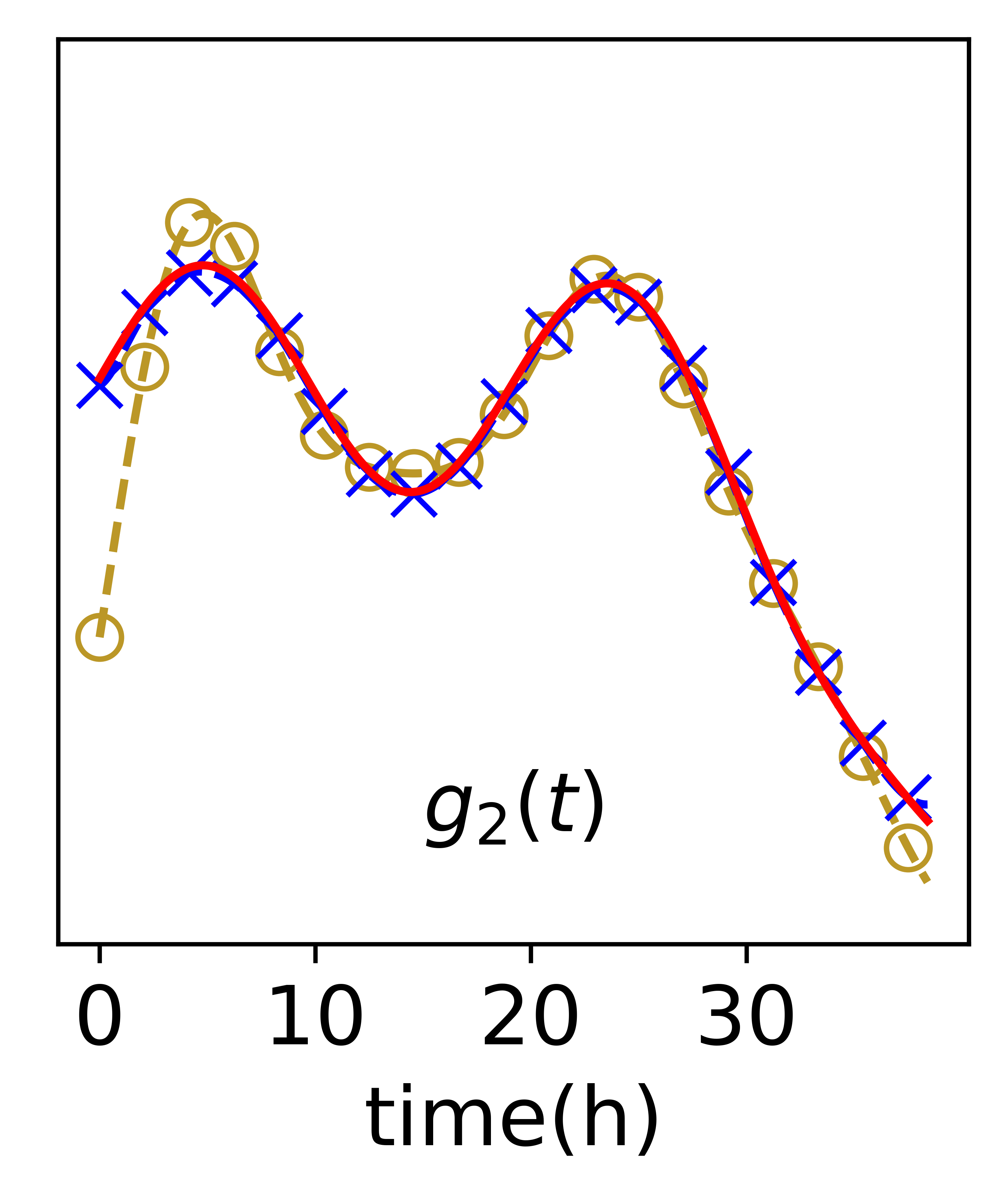}\label{fig: sub_figure2}}%
    \subfigure{\includegraphics[width=0.1925\hsize, height=0.28\hsize]{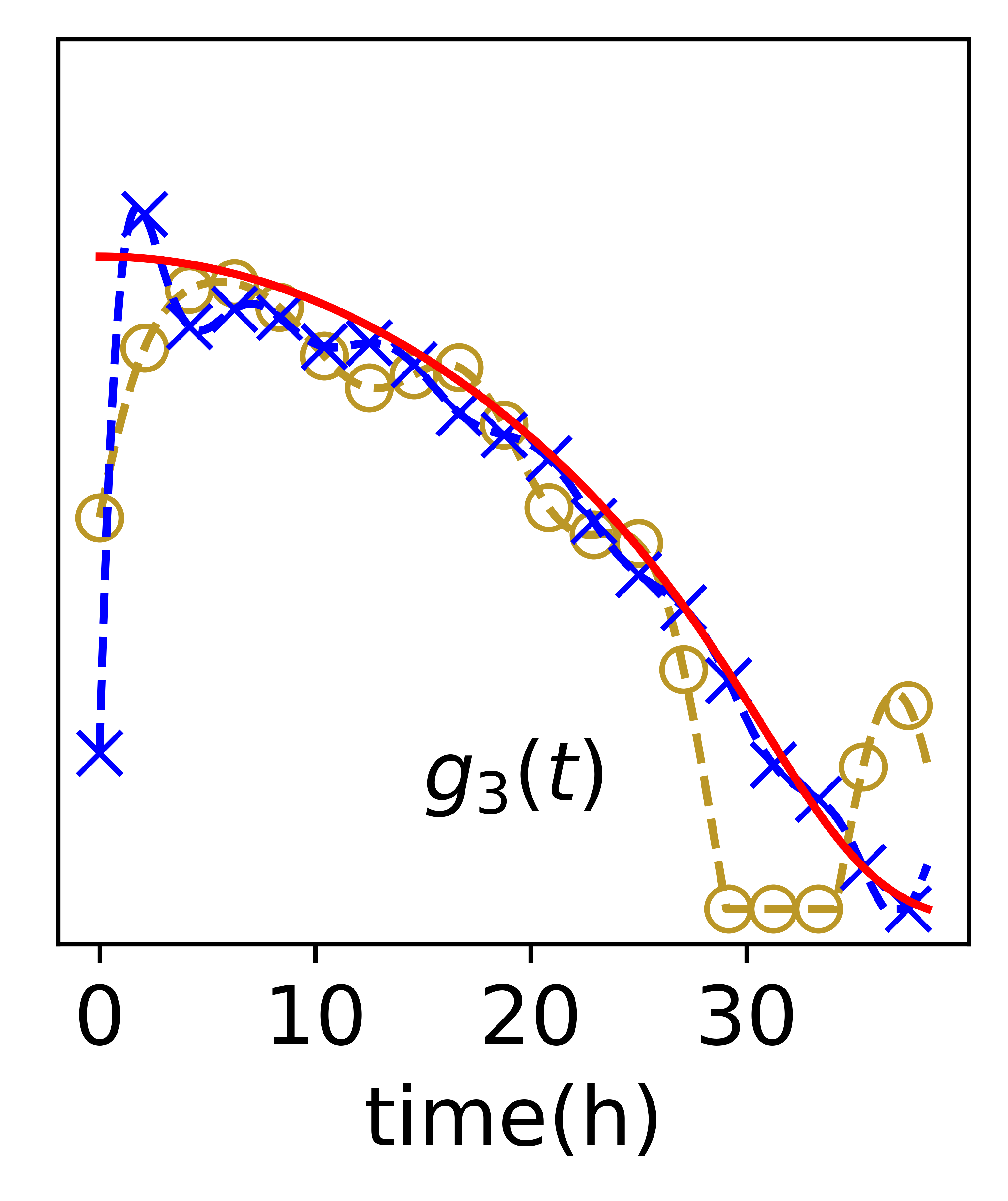}\label{fig: sub_figure3}}%
    \subfigure{\includegraphics[width=0.1925\hsize, height=0.28\hsize]{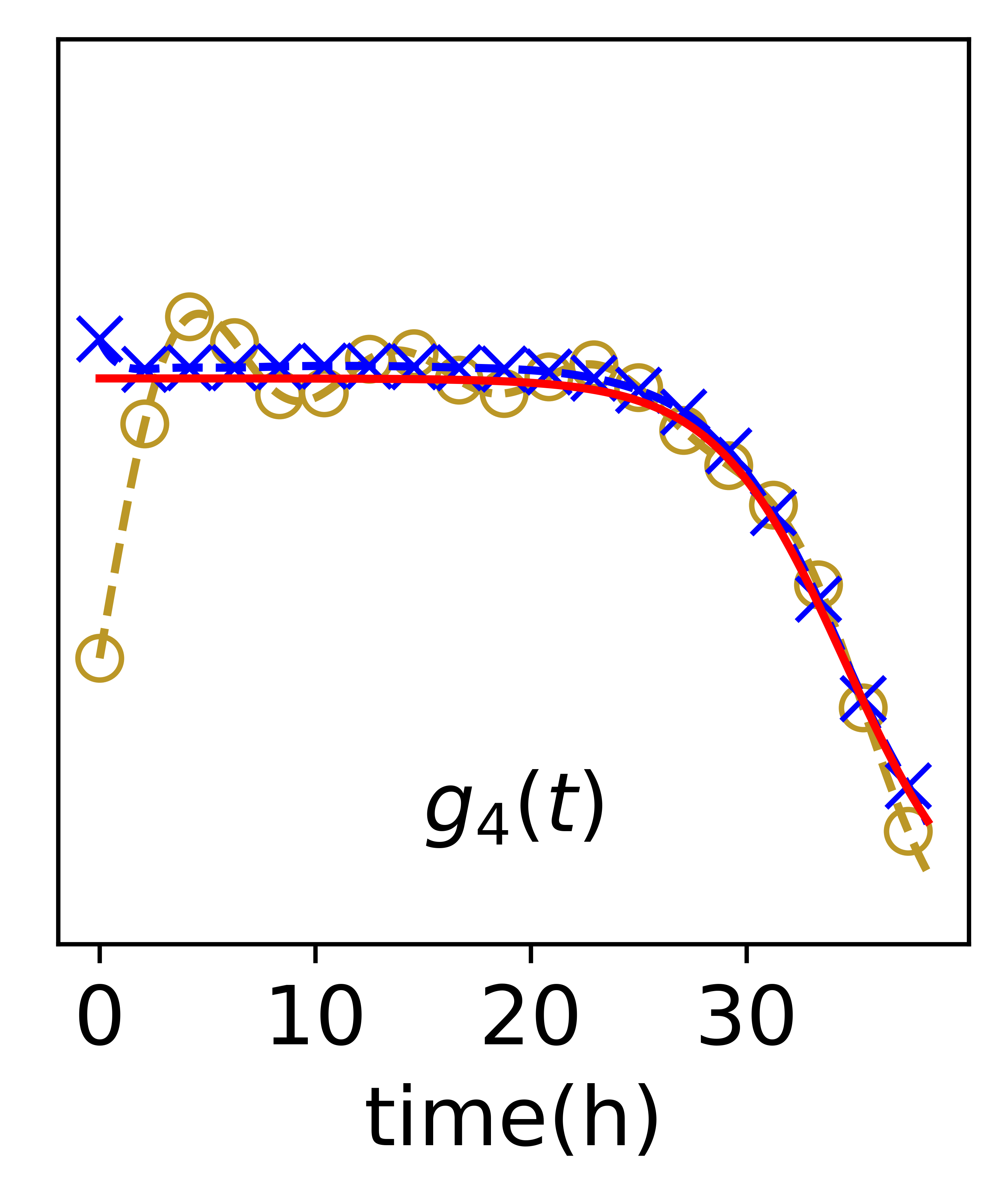}\label{fig: sub_figure4}}%
    \subfigure{\includegraphics[width=0.1925\hsize, height=0.28\hsize]{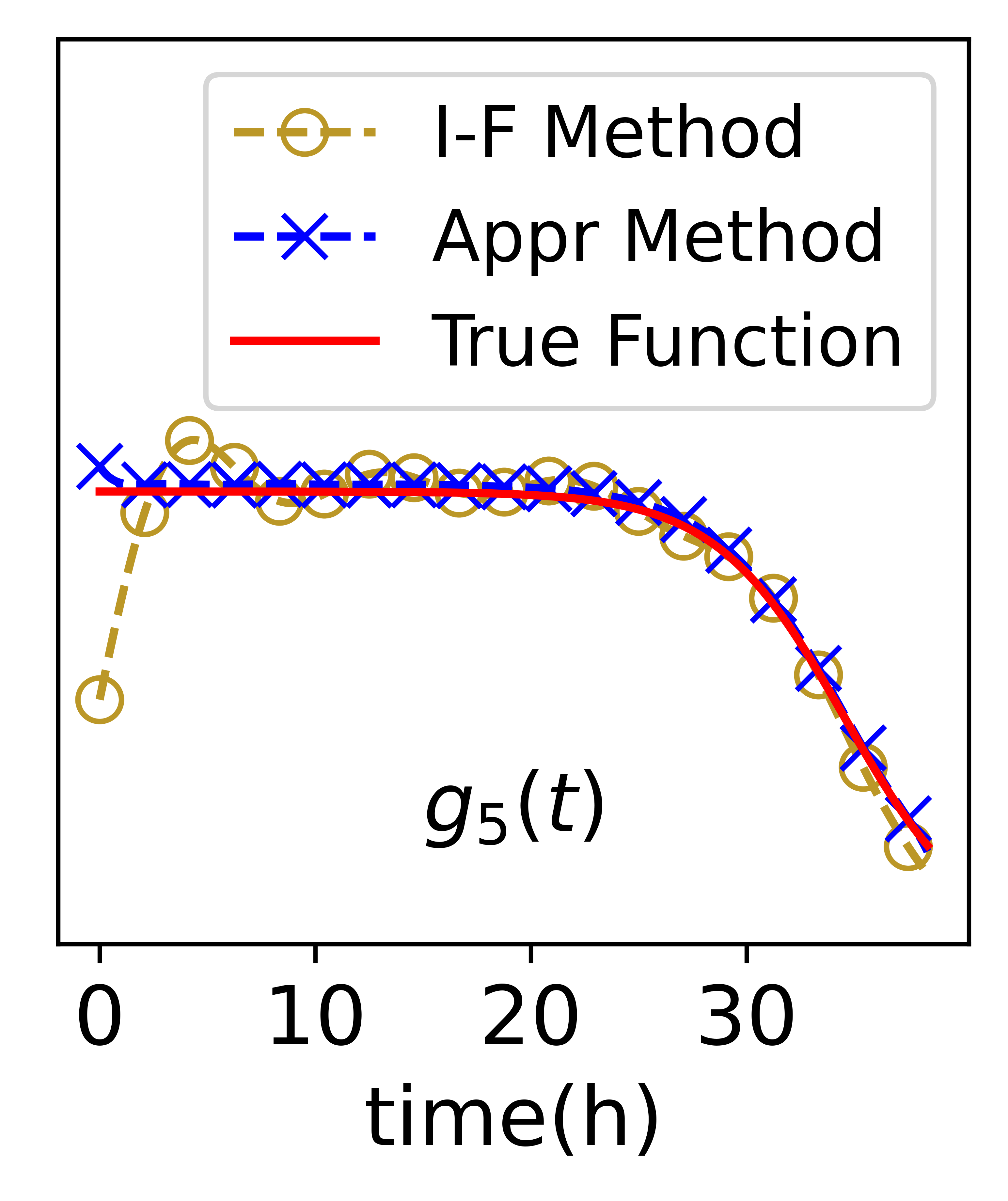}\label{fig: sub_figure5}}
    \caption{Reconstruction of intensity functions}
    \label{fig: figure5}
\end{figure}

\begin{table}[h]
    
    \centering
    \caption{Reconstruction error of intensity functions}
    \resizebox{8cm}{!}{
    \begin{tabular}{ccccccc} 
    \toprule 
    && \multicolumn{1}{c}{$g_1(t)$}& \multicolumn{1}{c}{$g_2(t)$}&\multicolumn{1}{c}{$g_3(t)$}&\multicolumn{1}{c}{$g_4(t)$}&\multicolumn{1}{c}{$g_5(t)$}\\
    \hline 
    \multicolumn{2}{c}{I-F Method }&8.97\% &8.36\% & 20.94\% &8.97\% & 8.02\%\\   
    \hline 
    \multicolumn{2}{c}{Appr Method }&3.30\% &0.79\% & 10.28\% &2.13\% & 1.66\%\\
    \bottomrule
    \end{tabular}}
    \label{table:intensity functions}
\end{table}

\subsection{Refinement of intensity function recovery}

Building upon the $5$-point scenario described in Section \ref{subsection:fourier inverse}, we employ an enhanced approximation method for refining the intensity function. In this method, denoted as Step (iv-1), the number of intensity basis functions is set to $L=8$. The results are presented in Figure \ref{fig: figure5}, under the label ``Appr Method", and the $L^2$ relative errors of the recovered intensity functions are tabulated in Table \ref{table:intensity functions}. Compared to the recovery results from Step (iii), the approximation method in Step (iv) show significantly better results, especially in the recovery of the function $g_3$. A noteworthy observation is the considerable reduction of the Gibbs phenomenon in this refined method, which leads to a substantial decrease in relative errors during the recovery process.

\section{Conclusion}\label{section:conclusion}

In this paper, we address a water pollution traceability problem based on the convection-diffusion-reaction equation.
By employing dynamic CGO solutions, we transform the inverse problem into its weak form. We subsequently establish the uniqueness of the inverse problem and showcase the stability of the zeroth Fourier coefficient through an estimation of the coefficient matrix's norm.

To recover the source term, our proposed numerical method consists of three parts. It begins with the computation of boundary controls tied to the employed test functions. Following this, we identify both the count and positions of the pollutant sources. To round off the process, we introduce two methods for the recovery of intensity functions: one that deduces them through their Fourier coefficients and another that sets up a minimization framework for their reconstruction. 
The first one, referred to as the inverse Fourier approach, drastically shortens online computational time, but offers a mediocre recovery of $g_j$'s. In contrast, the approximation technique, while being more computationally intensive online, shines in its ability to reduce recovery errors in a great deal. These two methods can be flexibly used in practice based on specific scenarios. Compared to previous approaches, our numerical method is capable of effectively recovering cases with a greater number of sources, and the results for recovering the source positions and intensity functions are relatively successful.

Nevertheless, our method presents certain challenges.
Persistent questions remain regarding the stability of all Fourier coefficients and the accuracy of the method in determining a larger number of sources. In the future, we aim to investigate the analytical questions that are crucial to inverse problems and employ our approach for more complex cases in the meantime. Further exploration of this approach and its applications pertaining to time-fractional equations is intriguing as well.

\bibliographystyle{elsarticle-num} 
\bibliography{ref}

\end{document}